\documentclass[a4paper,11pt,oneside]{article}
\usepackage{amsfonts,amsmath,amsthm,scalerel,verbatim,longtable,amssymb,makecell}
\usepackage{listings}
\usepackage{color,longtable}
\newcommand{\pres}[2]{\langle {#1}\ |\ {#2} \rangle}
\newcommand{\onetwoset}[3]{\Bigg \{ {#1}\ \Bigg |\
\begin{array}{l}
{#2}\\{#3}
\end{array}\Bigg \}}
\newcommand{\twolineset}[2]{\Bigg \{
\begin{array}{l}
{#1}\\{#2}
\end{array}\Bigg \}}
\newcommand{\Z}{\mathbb{Z}}
\newcommand{\N}{\mathbb{N}}

\newtheorem{theorem}{Theorem}
\newtheorem{lemma}[theorem]{Lemma}
\newtheorem{proposition}[theorem]{Proposition}
\newtheorem{corollary}[theorem]{Corollary}

\newtheorem{example}[theorem]{Example}
\newtheorem{conjecture}[theorem]{Conjecture}
\newtheorem{question}[theorem]{Question}

\allowdisplaybreaks
\numberwithin{theorem}{section}

\begin{document}
\title{An investigation into the cyclically presented groups with length three positive relators}
\author{Esamaldeen Mohamed\thanks{The first named author acknowledges financial support of the Libyan Government.}~~and Gerald Williams}

\maketitle

\begin{abstract}
We continue research into the cyclically presented groups with length three positive relators. We study small cancellation conditions, SQ-universality,  and hyperbolicity, we obtain the Betti numbers of the groups' abelianisations, we calculate the orders of the abelianisations of some groups, and we study isomorphism classes of the groups. Through computational experiments we assess how effective the abelianisation is as an invariant for distinguishing non-isomorphic groups.
\end{abstract}



\section{Introduction}\label{sec:intro}

The  \em cyclically presented group \em $G_n(w)$ is the group defined by the \em cyclic presentation \em
\begin{alignat}{1}
P_n(w)= \pres{x_0,\ldots ,x_{n-1}}{w, \theta(w), \ldots , \theta^{n-1}(w)}\label{eq:cycpres}
\end{alignat}
where $w = w(x_0, \ldots , x_{n-1})$ is a word in the free group $F_n$ with generators $x_0, \ldots ,x_{n-1}$
and $\theta : F_n \rightarrow F_n$ is the automorphism of $F_n$ given by $\theta(x_i) = x_{i+1}$ for
each $0\leq i< n$ (subscripts mod~$n$). In this article we continue investigations into the cyclically presented groups $\Gamma_n(k,l)$ defined by the cyclic presentations
\begin{alignat}{1}
P_n(k,l)=\pres{x_0,\ldots, x_{n-1}}{x_ix_{i+k}x_{i+l}\ (0\leq i< n)}\label{eq:GammaPres}
\end{alignat}
(where $0\leq k,l< n$, and subscripts are taken mod~$n$). This family of groups was introduced in~\cite{CRS} and studied in depth in~\cite{EdjvetWilliams} and further in~\cite{BogleyWilliams17}. (A reduced element of a free group is \em positive \em if all of its exponents are positive so these are precisely the cyclic presentations whose relators are positive words of length~3.) The related cyclically presented groups
\begin{alignat}{1}
G_n(m,k)=\pres{x_0,\ldots,x_{n-1}}{x_ix_{i+m}=x_{i+k}\ (0\leq i< n)}.\label{eq:Gnmk}
\end{alignat}
were studied in~\cite{CHR},\cite{BV},\cite{W},\cite{GilbertHowie},\cite{COS},\cite{HowieWilliams},\cite{HowieWilliams17} -- see~\cite{WRevisited} for a survey. In particular, Section~6 of~\cite{HowieWilliams} considers small cancellation conditions and SQ-universality for $G_n(m,k)$, and~\cite{COS} considers isomorphism classes of the groups $G_n(m,k)$. The purpose of this article, continuing work from~\cite{EdjvetWilliams},\cite{BogleyWilliams17}, is to provide similar theoretical and computational studies to these for the groups $\Gamma_n(k,l)$. We use GAP~\cite{GAP} to perform our experiments.

Easy reductions allow us to assume $(n,k,l)=1$, $k\neq 0,l\neq 0,k\neq l$ (see Section~\ref{sec:preliminaries}). Under these hypotheses, the article~\cite{EdjvetWilliams} obtained information about the groups $\Gamma_n(k,l)$ in terms of four true/false conditions (A),(B),(C),(D) defined in terms of congruences involving the parameters $n,k,l$ (see Section~\ref{sec:preliminaries} for definitions). It was shown in~\cite{EdjvetWilliams} that if none of the conditions (B),(C),(D) hold then the presentation~(\ref{eq:GammaPres}) satisfies the small cancellation condition C(3)-T(6), and hence (by~\cite{EdjvetHowie}) $\Gamma_n(k,l)$ contains a non-abelian free subgroup. In Section~\ref{sec:smallcancSQ} we continue this line of enquiry and show that the presentation~(\ref{eq:GammaPres}) does not satisfy any of stronger standard non-metric small cancellation conditions but that, for $n\neq 7$, if none of (B),(C),(D) hold then it does satisfy an alternative strengthening known as C(3)-T(6)-non-special, and hence (by~\cite{HowieSQ}) the group is SQ-universal. This suggests that many of the C(3)-T(6) groups $\Gamma_n(k,l)$ may in fact be hyperbolic, and we show that there are at most finitely many such groups that are not. The analysis to this point gives that the only case where the SQ-universality status is unknown is the case where (D) holds and (B) does not. In these cases $n$ is even and the group $\Gamma_n(k,l)$ can be shown to be isomorphic to $\Gamma_n(1,n/2-1)$. We show in Corollary~\ref{cor:Gamman1n/2-1n=8mod16} that $\Gamma_n(1,n/2-1)$ is large, and hence SQ-universal, when $(n,16)=8$.

The abelianisation $\Gamma_n(k,l)^\mathrm{ab}$ is infinite if and only if condition (A) holds (see~\cite[Lemma~2.2]{EdjvetWilliams}). In Section~\ref{sec:abelianisation} we improve this result to calculate the Betti number, or torsion-free rank, of $\Gamma_n(k,l)^\mathrm{ab}$. It follows from this that unlike, for example, the cyclically presented groups considered in~\cite{WilliamsLOG} the Betti number is not a useful tool for proving non-isomorphism amongst groups $\Gamma_n(k,l)$ with infinite abelianisation. In this section we also calculate $|\Gamma_n(1,n/2-1)^\mathrm{ab}|$, the order of the abelianisation.

For a fixed $n$, there are typically many isomorphisms amongst the groups $\Gamma_n(k,l)$. A special case of Problem~2.9 of~\cite{CRS} is to find a system of arithmetic conditions on the parameters that completely determines the isomorphism type of the group $\Gamma_n(k,l)$. In Sections~\ref{sec:n=p^aq^b}--\ref{sec:210} we investigate isomorphism classes of these groups. In Section~\ref{sec:n=p^aq^b} we consider isomorphism classes in the case when $n$ has at most two prime factors; in Section~\ref{sec:n<19} we apply this to determine precisely the isomorphism classes for $n\leq 29$. In Section~\ref{sec:ABCD} we analyse the general case $n\geq 19$ in terms of the four conditions (A),(B),(C),(D) and reduce the problem to consideration of the cases where none of (B),(C),(D) hold. In Section~\ref{sec:FFFF} we consider the case when none of (A),(B),(C),(D) hold and, when $n$ has certain divisors, we use properties of the groups' abelianisations to provide sets of pairwise non-isomorphic groups. In Section~\ref{sec:210} we consider groups where none of (B),(C),(D) hold and, for $n<210$, we assess how effective the abelianisation is as an invariant for distinguishing non-isomorphic groups.

\section{Preliminaries}\label{sec:preliminaries}

If the greatest common divisor $d=(n,k,l)>1$ then $\Gamma_n(k,l)$ is the free product of $d$ copies of (the non-trivial group) $\Gamma_{n/d}(k/d,l/d)$ (see~\cite[Lemma~2.4]{CRS}) and, with $(n,k,l)=1$, if $k=0$ or $l=0$ or $k=l$ then $\Gamma_n(k,l)$ is a finite cyclic group. Therefore it suffices to consider these groups under the hypotheses $n\geq 3$, $1\leq k,l< n$, $k\neq l$, $(n,k,l)=1$. The (A),(B),(C),(D) conditions alluded to earlier are defined as follows:
\begin{alignat*}{1}
A(n,k,l)
&=\begin{cases}
  T&\mathrm{if}~n\equiv 0~\mathrm{and}~k+l\equiv0~\mathrm{mod}~3,\\
  F&\mathrm{otherwise},
\end{cases}\\
B(n,k,l)
&=\begin{cases}
  T&\mathrm{if}~(k+l)\equiv 0~\mathrm{or}~(2l-k)\equiv0~\mathrm{or}~(2k-l)\equiv 0~\mathrm{mod}~n,\\
  F&\mathrm{otherwise},
\end{cases}\\
C(n,k,l)
&=\begin{cases}
  T&\mathrm{if}~3l\equiv 0~\mathrm{or}~3k\equiv 0~\mathrm{or}~3(l-k)\equiv 0~\mathrm{mod}~n,\\
  F&\mathrm{otherwise},
\end{cases}\\
D(n,k,l)
&=\begin{cases}
  T&\mathrm{if}~2(k+l)\equiv 0~\mathrm{or}~2(2l-k)\equiv 0~\mathrm{or}~2(2k-l)\equiv 0~\mathrm{mod}~n,\\
  F&\mathrm{otherwise}.
\end{cases}
\end{alignat*}
We shall sometimes abuse notation and say, for example, that condition (A) holds (resp. does not hold) if $A(n,k,l)=T$ (resp. $A(n,k,l)=F$).

We summarize key algebraic results obtained in~\cite{EdjvetWilliams},\cite{BogleyWilliams17} in terms of the three conditions (A),(B),(C) in Table~\ref{tab:EdjvetWilliams}. (Note that if (B) and (C) hold then (A) holds, so there is no $(F,T,T)$ entry.) In this table $\alpha=3(2^{n/3}-(-1)^{n/3})$, $\gamma=(2^{n/3}-(-1)^{n/3})/3$ and `large' denotes a \em large \em group, i.e.\,a group that has a finite index subgroup that maps onto the free group of rank~2~\cite{BaumslagPride},\cite{Gromov}. A group $G$ is \em SQ-universal \em if every countable group can be embedded in a quotient of $G$. Large groups are SQ-universal and therefore contain a non-abelian free subgroup.

Table~\ref{tab:EdjvetWilliams} is based on~\cite[Table~1]{EdjvetWilliams} but with a few changes. The first concerns the group in entry~2. In~\cite[Table~1]{EdjvetWilliams} this group was described as metacyclic and its precise identification was conjectured in~\cite[Conjecture~3.4]{EdjvetWilliams}. That conjecture was subsequently proved in~\cite[Corollary~D]{BogleyWilliams17}, so we now include the precise identification. The notation $B(M,N,r,\lambda)$ refers to the metacyclic group defined by the presentation
\[\pres{a,b}{a^M=1, bab^{-1}=a^r, b^N=a^{\lambda M/(M,r-1)}}\]
where $r^N\equiv 1$~mod~$M$. Any finite metacyclic group $L$ with a metacyclic extension $\Z_M \hookrightarrow L\twoheadrightarrow \Z_N$ has a presentation of this form (see \cite[Chapter IV.2]{BeylTappe} or \cite[Chapter~3]{Johnson}).

\begin{table}
    \begin{tabular}{|cc|c|c|c|c|}
     \hline
     $(A,B,C)$ & &$\Gamma=\Gamma_n(k,l)$   &$\Gamma^\mathrm{ab}$ & def$(\Gamma)$ \\\hline
     $(F,F,F)$ & & infinite and torsion-free      &finite~$\neq 1$& 0 \\\hline
     $(F,F,T)$ & & $B\left((2^n-(-1)^n)/3,3,2^{2n/3},1\right)$ & $\Z_\alpha$ & 0 \\\hline
     $(F,T,F)$  & & $\Z_3$   &$\Z_3$ & 0 \\\hline
     $(T,F,F)$  & $n\neq18 $ &  Large & $\Z^2\oplus$~finite   &0 \\\hline
     $(T,F,F)$  & $n=18 $  &  $\Z*\Z*\Z_{19}$    & $\Z^2\oplus\Z_{19}$ & 2 \\\hline
     $(T,F,T)$  & & $\Z*\Z*\Z_{\gamma}$    & $\Z^2\oplus\Z_{\gamma}$ & 2 \\ \hline
     $(T,T,F)$  & & $\Z*\Z$     & $\Z^2$ & 2 \\\hline
     $(T,T,T)$  & & $\Z*\Z$    & $\Z^2$ & 2 \\
      \hline
   \end{tabular}%
\caption{Groups $\Gamma_n(k,l)$ with $n\geq 3$, $(n,k,l)=1$, $k\neq l$, by (A),(B),(C) conditions~(\cite{EdjvetWilliams},\cite{BogleyWilliams17}).\label{tab:EdjvetWilliams}}
\end{table}

The second change is to entry~4. In~\cite[Table~1]{EdjvetWilliams} the entry $\Gamma^\mathrm{ab}$ was marked as `infinite'. However, in light of Theorem~\ref{thm:betti} (to be proved in Section~\ref{sec:abelianisation}) we now know that in these cases $\Gamma_n(k,l)^\mathrm{ab}\cong \Z^2 \oplus A_0$, for some finite abelian group $A_0$; we indicate this in the table as `$\Z^2\oplus$~finite'.

Given a finite presentation $P=\pres{X}{R}$, the \em deficiency \em of $P$, $\mathrm{def}(P)=|X|-|R|$. The \em deficiency \em of a group $G$, $\mathrm{def}(G)$, is defined to be the maximum of the deficiencies of all finite presentations of $G$. An aspherical presentation $P$ of any given group $G$ has $\mathrm{def}(P)=\mathrm{def}(G)$~\cite[page~478]{Trotter}. The aspherical presentations~(\ref{eq:GammaPres}) were classified in~\cite{EdjvetWilliams}, so groups defined by these presentations have deficiency zero. (Here, by an \em aspherical \em presentation, we mean one whose presentation complex has trivial second homotopy group.) Since the deficiency will be a useful group invariant to us, our third change is to replace the `aspherical' column of~\cite[Table~1]{EdjvetWilliams} by the `$\mathrm{def}(\Gamma)$' column in Table~\ref{tab:EdjvetWilliams}. Since groups defined by aspherical presentations in which no relator is a proper power are torsion-free (\cite{Huebschmann}) and (by~\cite{EdjvetWilliams}) the presentation $P_n(k,l)$ is aspherical when none of (A),(B),(C) hold, we also add `torsion-free' to entry~1 of Table~\ref{tab:EdjvetWilliams}.

In order to study isomorphism classes of groups $\Gamma_n(k,l)$, for each $n\geq 3$ we introduce the sets
\[S(n)=\{\Gamma_n(k,l)\ |\ 1\leq k,l< n, k\not=l, (n,k,l)=1\}\]
(where isomorphic groups are identified) and we let $f(n)=|S(n)|$. For example
\[S(4)=\{\Gamma_4(1,2), \Gamma_4(1,3), \Gamma_4(2,1), \Gamma_4(2,3), \Gamma_4(3,1) ,\Gamma_4(3,2)\}\]
but $\Gamma_4(1,2)\cong \Gamma_4(1,3)\cong \Gamma_4(2,1)\cong \Gamma_4(2,3)\cong \Gamma_4(3,1)\cong\Gamma_4(3,2)\cong\Z_3$ (since in each case (B) holds and neither (A) nor (C) hold) so $S(4)=\{\Gamma_4(1,2)\}=\{\Z_3\}$ and $f(4)=1$. We make extensive use of the following lemma (which is essentially \cite[Lemma~2.1]{EdjvetWilliams}) which establishes isomorphisms amongst the groups $\Gamma_n(k,l)$:

\begin{lemma}[{\cite[Lemma~2.1]{EdjvetWilliams}}]\label{lem:gammaiso}
Let $1\leq k, l< n$ then
\begin{enumerate}
  \item[(i)] $\Gamma_n(k,l) \cong \Gamma_n(l-k,-k)$;
  \item[(ii)] $\Gamma_n(k,l) \cong \Gamma_n(l,k)$;
  \item[(iii)] $\Gamma_n(k,l) \cong \Gamma_n(k-l,-l)$;
  \item[(iv)] $\Gamma_n(k,l) \cong \Gamma_n(k,k-l)$;
  \item[(v)] if $(\psi, n) = 1$ then $\Gamma_n(k,l) \cong \Gamma_n(\psi k , \psi l)$;
\end{enumerate}
(where the bracketed terms are taken mod~$n$).
\end{lemma}

The case where (D) holds and (B) does not will play an important role in our analysis. In the following lemma we note that this corresponds to the group $\Gamma_n(1,n/2-1)$.

\begin{lemma}\label{lem:FFFT=GN(1,n/2-1)}
Suppose $n\geq 3$, $1\leq k,l< n$, $k\neq l$, $(n,k,l)=1$, and that (D) holds and (B) does not hold. Then $\Gamma_n(k,l)\cong\Gamma_n(1,n/2-1)$.
\end{lemma}

\begin{proof}
If (D) holds and (B) does not then one of $k+l$, $2l-k$, or $2k-l$ is equivalent to $n/2$ mod~$n$. Therefore $\Gamma_n(k,l)$ is isomorphic to one of $\Gamma_n(k,n/2-k)$, $\Gamma_n(2l+n/2,l)$, or $\Gamma_n(k,n/2+2k)$ and since $\Gamma_n(2l+n/2,l)\cong \Gamma_n(l,2l+n/2)$ and $\Gamma_n(k,2k+n/2)\cong \Gamma_n(k,n/2-k)$ by parts (ii) and (iv) of Lemma~\ref{lem:gammaiso} it suffices to consider the group $\Gamma_n(k,n/2-k)$. Again by part (ii) of Lemma~\ref{lem:gammaiso} we have $\Gamma_n(k,n/2-k)\cong \Gamma_n(n/2-k,k)=\Gamma_n(k',n/2-k')$ where $k'=n/2-k$. The condition $(n,k,l)=1$ implies that $(n/2,k)=1$ so if $k$ is even then $k'$ is odd, so without loss of generality we may assume that $k$ is odd. Since $(n,k)=1$ the result then follows from Lemma~\ref{lem:gammaiso}(v).
\end{proof}

\section{Small cancellation, SQ-universality and hyperbolicity}\label{sec:smallcancSQ}

Let $P_n(k,l)$ denote the cyclic presentation~(\ref{eq:GammaPres}). The following theorem and corollary (which follows from~\cite{EdjvetHowie}) were proved in~\cite{EdjvetWilliams}:

\begin{theorem}[{\cite[Lemma~5.1]{EdjvetWilliams}}]\label{thm:C3T6}
Let $n\geq 3$, $1\leq k,l< n$, $k\neq l$ and $(n,k,l)=1$. Then the presentation $P_n(k,l)$ satisfies the small cancellation condition C(3)-T(6) if and only if none of (B),(C),(D) hold.
\end{theorem}

\begin{corollary}[{\cite[Corollary~5.2]{EdjvetWilliams}}]\label{cor:freesubgroups}
Let $n\geq 3$, $1\leq k,l< n$, $k\neq l$ and $(n,k,l)=1$. If none of (B),(C),(D) hold then $\Gamma_n(k,l)$ contains a non-abelian free subgroup.
\end{corollary}

We now record some other consequences which, as in the proof of \cite[Corollary~11]{HowieWilliams} (which concerns the groups $G_n(m,k)$), follow from~\cite[Chapter~V, Theorems 6.3 and 7.6]{LyndonSchupp}, \cite[Section~3]{ElMosalamyPride}), \cite[Theorem~2]{GerstenShortII}.

\begin{corollary}\label{cor:C3T6properties}
Let $n\geq 3$, $1\leq k,l< n$, $k\neq l$ and $(n,k,l)=1$. If none of (B),(C),(D) hold then $\Gamma_n(k,l)$ has solvable word and conjugacy
problems, is automatic, and acts properly and cocompactly on a CAT(0) space.
\end{corollary}

Stronger conclusions could be drawn if the presentation $P_n(k,l)$ were to satisfy C(3)-T(7). However, we now show that C(3)-T(6) is the strongest of the standard non-metric small cancellation conditions that $P_n(k,l)$ can satisfy. (This is in contrast to the situation for the groups $G_n(m,k)$ where C($3$)-T($7$) can be satisfied~\cite[Theorem~10]{HowieWilliams}.) As in~\cite{EdjvetWilliams} we use the characterisation~\cite{HillPrideVella} of the C($3$)-T($q$) conditions in terms of the \em star graph \em of a presentation $\pres{X}{R}$. (The star graph is also called the \em co-initial graph \em or \em star complex.\em) This is the graph $\Lambda$ with vertex set $X\cup X^{-1}$ and with an edge from $x$ to $y$ for each distinct word of the form $x^{-1}yu$ ($x\neq y$) that is a cyclic permutation of a relator or its inverse~\cite[page~61]{LyndonSchupp}. (The edges occur in inverse pairs.) By~\cite{HillPrideVella} a presentation $\pres{X}{R}$ in which each relator has length at least 3 satisfies C($3$)-T($q$) ($q > 4$) if and only if $\Lambda$ has no cycle of length less than $q$. (See also~\cite{ElMosalamyPride} for further explanation.)
The star graph $\Lambda_n(k,l)$ of the presentation $\Gamma_n(k,l)$ is the (bipartite) graph with vertices $x_i,x_i^{-1}$ and edges $(x_i,x_{i+k}^{-1})$, $(x_i,x_{i+l-k}^{-1})$, $(x_i,x_{i-l}^{-1})$ ($0\leq i< n$, subscripts mod~$n$).

\begin{theorem}\label{thm:neverC3T7}
Let $n\geq 3$, $1\leq k,l< n$, $k\neq l$ and $(n,k,l)=1$. Then the presentation $P_n(k,l)$ is not C($p$)-T($q$) for any $p>3$ or $q>6$.
\end{theorem}

\begin{proof}
A presentation satisfies the condition C($p$) if and only if none of its relators can be written as a product of fewer than $p$ pieces. Since each generator of $P_n(k,l)$ is a piece and all relators have length~$3$, every relator can be written as a product of 3 pieces, so the presentation $P_n(k,l)$ cannot satisfy C($p$) for $p>3$.

Suppose that $P_n(k,l)$ satisfies T($q$) for some $q\geq 6$. Then by~\cite[Section~5]{PrideDependenceProblem} every piece has length 1 so $P_n(k,l)$ also satisfies C($3$) and therefore satisfies C($3$)-T($6$).
Theorem~\ref{thm:C3T6} then implies that (B) does not hold and so for all $1\leq k,l< n$ the star graph $\Lambda_n(k,l)$ contains the 6-cycle $x_0,x_{k}^{-1},x_{2k-l},x_{2k-2l}^{-1},x_{k-2l},x_{n-l}^{-1},x_0$. Thus, using the C($3$)-T($q$) characterisation above, $q=6$, as required.
\end{proof}

An alternative strengthening of the C(3)-T(6) condition is the C(3)-T(6)-non-special condition, introduced in~\cite{HowieSQ} (see also~\cite{EdjvetVdovina}). With a few known exceptions, groups defined by C(3)-T(6)-non-special presentations are SQ-universal~\cite{HowieSQ}.
A C($3$)-T($6$) presentation is \em special \em if every relator has length 3 and $\Lambda$ is isomorphic to the incidence graph of a finite projective plane (and \em non-special \em otherwise). (In~\cite{EdjvetVdovina} these special presentations are called \em $(3,3)$-special presentations.\em ) As explained in the proof of~\cite[Theorem~2(iii)]{EdjvetVdovina} special presentations are exactly the triangle presentations of~\cite{CMSZ}. Answering~\cite[Question~6.11]{HowieSQ} Theorem~2(iii) of~\cite{EdjvetVdovina} gives that groups defined by C($3$)-T($6$)-special presentations are not SQ-universal.

\begin{theorem}\label{thm:C3T6special}
Let $n\geq 3$, $1\leq k,l< n$, $k\neq l$. Then the presentation $P_n(k,l)$ satisfies the small cancellation condition C(3)-T(6)-special if and only if $n=7$ and none of (B),(C),(D) hold.
\end{theorem}

\begin{proof}
If $P_n(k,l)$ satisfies the small cancellation condition C(3)-T(6)-special then the star graph $\Lambda$ is the incidence graph of a projective plane $P$. Then since the vertices of $\Lambda$ have degree~3 it follows that $P$ has order 2 so has 7 points and 7 lines (see for example~\cite[Section~16.7 and Exercise~16.10.10]{Biggs}) so~$\Lambda$ has precisely $14$ vertices and hence $n=7$. Thus $(n,k,l)=1$, so $\Gamma_n(k,l)$ is C(3)-T(6)-special if and only if none of (B),(C),(D) hold, by Theorem~\ref{thm:C3T6}.

If $n=7$ and none of (B),(C),(D) hold then it is straightforward to check that $\Lambda$ is the Heawood graph (which is the incidence graph of the 7-point projective plane).
\end{proof}

\begin{corollary}\label{cor:-FFFSQuniversal}
Let $n\geq 3$, $1\leq k,l< n$, $k\neq l$ and $(n,k,l)=1$ and suppose that none of (B),(C),(D) hold. Then $\Gamma_n(k,l)$ is SQ-universal if and only if $n\neq 7$
\end{corollary}

\begin{proof}
If the presentation is C(3)-T(6)-non-special then $\Gamma_n(k,l)$ is SQ-\linebreak universal by~\cite{HowieSQ}; if it is C(3)-T(6)-special then it is not SQ-universal by~\cite[Theorem~2(iii)]{EdjvetVdovina}.
\end{proof}

If it could be shown that the C(3)-T(6) groups of Corollary~\ref{cor:-FFFSQuniversal} are hyperbolic (and hence non-elementary hyperbolic, by Corollary~\ref{cor:freesubgroups}) SQ-universality would also follow from~\cite{Olshanskii},\cite{Delzant}. We are grateful to an anonymous referee for providing the following result which shows that, for all but finitely many~$n$, the C(3)-T(6) groups $\Gamma_n(k,l)$ are hyperbolic.

\begin{theorem}\label{thm:allbutfinitelyhyperbolic}
Let $k,l$ be fixed integers with $k\neq 0$, $l\neq 0$, $k\neq l$, $k+l\neq 0$, $2l-k\neq 0$, $2k-l\neq 0$. Then there exists $N\in \mathbb{N}$ such that for all $n\geq N$ with $(n,k,l)=1$ none of (B),(C),(D) hold and $\Gamma_n(k,l)$ is hyperbolic.
\end{theorem}

\begin{proof}
Note first that the hypotheses imply that for sufficiently large $n$ we have $B(n,k,l)=C(n,k,l)=D(n,k,l)=F$. Now either Theorem~3 or Theorem~4 of~\cite{IvanovSchupp} imply that the one-relator group $E(k,l)=\pres{x,t}{xt^kxt^{l-k}xt^{-l}}$ is hyperbolic, and it follows from~\cite[Th\'{e}or\`{e}me~I]{Delzant} or~\cite[Theorem~3]{OlshanskiiIJAC93} that there exists an $N\in \N$ such that for all $n\geq N$ the quotient $E_n(k,l)=\pres{x,t}{t^n,xt^kxt^{l-k}xt^{-l}}$ is hyperbolic. The kernel of the epimorphism \linebreak $E_n(k,l)\rightarrow \pres{t}{t^n}$ given by $t\mapsto t$, $x\mapsto 1$ is the group $\Gamma_n(k,l)$. Thus, since it is a finite index subgroup of $E_n(k,l)$, the group $\Gamma_n(k,l)$ is hyperbolic, as required.
\end{proof}

We now consider the case $n=7$ and none of (B),(C),(D) holds. It follows from Lemma~\ref{lem:gammaiso} that $\Gamma_n(k,l)\cong \Gamma_7(1,3)$. The following example demonstrates that not all C(3)-T(6) groups $\Gamma_n(k,l)$ are hyperbolic.

\begin{example}[The group $\Gamma_7(1,3)$]\label{ex:G_7(1,3)}
\em The group $\Gamma=\Gamma_7(1,3)$ appeared in~\cite[Example~3.3]{EdjvetHowie} as the group $G_3$ (and later in~\cite[Example~6.3]{HowieSQ}) as an example of one with a C(3)-T(6)-special (cyclic) presentation. The proof of~\cite[Theorem~2(iii)]{EdjvetVdovina} shows that groups defined by C($3$)-T($6$)-special presentations are \em just-infinite \em (i.e.\,they are infinite but all their proper quotients are finite) and hence $\Gamma$ is not SQ-universal. Further, $\Gamma$ contains a non-abelian free subgroup (by Corollary~\ref{cor:freesubgroups}) so if it is hyperbolic then it is non-elementary hyperbolic. But by~\cite{Olshanskii},\cite{Delzant} non-elementary hyperbolic groups are SQ-universal, so $\Gamma$ is not hyperbolic. The group $\Gamma$ also appears in~\cite[Section~4]{CMSZ} and as the group $A.1$ in~\cite[Section~4]{CMSZ2} in connection with buildings. In~\cite[Section~3]{BarkerBostonPeyerimhoffVdovina} it is observed that $\Gamma$ acts via covering transformations on a thick Euclidean building of type $\tilde{A}_2$, and it was used to construct the first known infinite family of groups with mixed Beauville structures. \em
\end{example}

Returning to the properties of SQ-universality and the existence of non-abelian free subgroups, from Table~\ref{tab:EdjvetWilliams}, Corollary~\ref{cor:freesubgroups} and Corollary~\ref{cor:-FFFSQuniversal}, we see that the only case where these remain unresolved is the case when none of (A),(B),(C) hold, but (D) does hold. By Lemma~\ref{lem:FFFT=GN(1,n/2-1)}, and as was observed without proof in~\cite[page~774]{EdjvetWilliams}, this is the case $(n,6)=2$ and $\Gamma_n(k,l)\cong\Gamma_n(1,n/2-1)$. We therefore pose the following question:

\begin{question}\label{q:IsG_n(1,n/2-1)SQfree}
\em Suppose $n\geq 8$ and $(n,6)=2$.
\begin{itemize}
  \item[(a)] Is $\Gamma_n(1,n/2-1)$ SQ-universal?
  \item[(b)] Does $\Gamma_n(1,n/2-1)$ contain a non-abelian free subgroup?\em
\end{itemize}
\end{question}

\begin{example}[The group $\Gamma_8(1,3)$]\label{ex:G_8(1,3)}
\em As explained in~\cite[Example~3]{BogleyWilliams17} for $n=8$ the group $\Gamma=\Gamma_n(1,n/2-1)=\Gamma_8(1,3)$ contains the direct product of two copies of the free group of rank~4 as an index~9 subgroup, so it is large, and hence SQ-universal.
The group $\Gamma$ therefore contains $F_2\oplus F_2$ as a subgroup so by \cite{BigdelyWise},\cite[Theorem~9.3.1]{BigdelyThesis} there is no C(3)-T(6) or C(6) presentation of $\Gamma$ and since it contains $\Z\oplus \Z$ it is not hyperbolic. Further, we have $\Gamma_8(1,3)^\mathrm{ab}\cong \Z_3\oplus \Z_3\oplus \Z_3$.
(We note that in~\cite[page~774]{EdjvetWilliams} the index~9 subgroup mentioned above was erroneously stated to be the free abelian group of rank~8.)
\end{example}

For a group $G$ we write $d(G)$ to denote the minimum number of generators of $G$. As a corollary to Example~\ref{ex:G_8(1,3)} we have the following:

\begin{corollary}\label{cor:Gamman1n/2-1n=8mod16}
Suppose $(n,16)=8$. Then
\begin{itemize}
  \item[(a)] $\Gamma_n(1,n/2-1)$ is large, and hence SQ-universal; and
  \item[(b)] $d(\Gamma_n(1,n/2-1)^\mathrm{ab})\geq 3$.
\end{itemize}
\end{corollary}

\begin{proof}
Here $n=8+16t$ for some $t\geq 0$ and so
\[\Gamma_n(1,n/2-1)=\pres{x_0,\ldots, x_{7+16t}}{x_ix_{i+1}x_{i+3+8t}\ (0\leq i\leq 7+16t)}\\
\]
which maps onto
\begin{alignat*}{1}
Q&=\pres{x_0,\ldots, x_{7+16t}}{x_ix_{i+1}x_{i+3+8t},x_i=x_{i+8}\ (0\leq i\leq 7+16t)}\\
&=\pres{x_0,\ldots, x_7}{x_ix_{i+1}x_{i+3}\ (0\leq i\leq 7)}\\
&=\Gamma_8(1,3),
\end{alignat*}
and since $\Gamma_8(1,3)$ is large and $d(\Gamma_8(1,3)^\mathrm{ab})=3$ the result follows.
\end{proof}

\begin{example}[The group $\Gamma_{20}(1,9)$]\label{ex:G_20(1,9)}
\em Button reports~\cite{ButtonEmail} that the Magma routines described in~\cite{ButtonLargeComputer} show that $\Gamma_{20}(1,9)$ is large (but are inconclusive for $\Gamma_{10}(1,4)$ and $\Gamma_{14}(1,6)$). As in the proof of Corollary~\ref{cor:Gamman1n/2-1n=8mod16}, it follows that if $n\equiv 20$~mod~$40$ then $\Gamma_n(1,n/2-1)$ is large.\em
\end{example}

\section{Abelianisations}\label{sec:abelianisation}

Consider a cyclically presented group $G=G_n(w)$ with abelianisation \linebreak $G_n(w)^\mathrm{ab} \cong A_0\oplus \Z^\beta$, where $A_0$ is a finite abelian group and $\beta=\beta(G_n(w)^\mathrm{ab})$ is the \em Betti number \em (or \em torsion-free rank\em) of $G_n(w)^\mathrm{ab}$.
For each $0\leq i<n$ let $a_i$ denote the exponent sum of $x_i$ in $w=w(x_0,\ldots,x_{n-1})$, and let $f(t)=\sum_{i=0}^{n-1}a_it^i$, $g(t)=t^n-1$. It is well known that
\begin{alignat}{1}
|G_n(w)^\mathrm{ab})|=|\prod_{g(\lambda)=0} f(\lambda)|\label{eq:Gaborder}
\end{alignat}
when this product is non-zero, and $G_n(w)^\mathrm{ab}$ is infinite (i.e.\,$\beta(G_n(w)^\mathrm{ab})>0$) otherwise (\cite{Johnson}). Moreover, as was pointed out in~\cite{WilliamsLOG}, it follows from~\cite[Proposition~1.1]{Ingleton} or~\cite[Theorem~1]{Newman} that
\begin{alignat}{1}
\beta(G_n(w)^\mathrm{ab}) = \mathrm{deg}(\mathrm{gcd} (f(t),g(t))).\label{eq:betti}
\end{alignat}
We use this to obtain the Betti number of $\Gamma_n(k,l)^\mathrm{ab}$. For any $m\geq 1$ we shall use the notation $\zeta_m=e^{2\pi i /m}$.

\begin{theorem}\label{thm:betti}
Suppose $n\geq 3$, $1\leq k,l <n $, $k\neq l$, $(n,k,l)=1$. Then
\[
  \beta(\Gamma_n(k,l)^\mathrm{ab})=\begin{cases} 2&\mathrm{if~(A)~holds,}\\
  0&\mathrm{otherwise}.
\end{cases}\]
\end{theorem}

\begin{proof}
If (A) does not hold then $G_n(w)^\mathrm{ab}$ is finite, and hence $\beta(\Gamma_n(k,l)^\mathrm{ab})=0$, so assume that (A) holds.
We have $f(t)=1+t^k+t^l$, $g(t)=t^n-1$, $n=3m$ (some $m\geq 1$), $k+l\equiv 0$~mod~$3$, $k,l\not \equiv 0$~mod~$3$. Let $d=(k,l)$, $k_1=k/d, l_1=l/d$, so $(d,3)=1$, $k_1+l_1\equiv 0$~mod~$3$. By~\cite[Theorem~3]{Selmer},\cite[Theorem~3]{Ljunggren} $f(t)=(t^{2d}+t^d+1)h(t)$ where $h(t)$ has no roots of modulus~1. Therefore gcd$(f(t),g(t))=(t^{2d}+t^d+1,t^n-1)$. Let $\lambda$ be a root of $t^{2d}+t^d+1$ and of $t^n-1$. Then $\lambda^n=1$ and $\lambda^{2d}+\lambda^d+1=0$ so $\lambda^d\neq 1$ and $(\lambda^d-1)(\lambda^{2d}+\lambda^d+1)=0$, i.e.\,$\lambda^{3d}=1$. Thus $\lambda^{(n,3d)}=1$ so $\lambda^{3(m,d)}=1$. But $(m,d)=(m,k,l)=(n,k,l)=1$ so $\lambda^3=1$ and since $\lambda\neq 1$ we have $\lambda=\zeta_3$ or $\zeta_3^2$. Hence gcd$(f(t),g(t))=(t-\zeta_3)(t-\zeta_3^2)=t^2+t+1$ and so by (\ref{eq:betti}) $\beta(\Gamma_n(k,l)^\mathrm{ab})=\deg (t^2+t+1)=2$, as required.
\end{proof}

As observed in Section~\ref{sec:smallcancSQ} and Lemma~\ref{lem:FFFT=GN(1,n/2-1)} the group $\Gamma_n(1,n/2-1)$ (where $n$ is even) will play an important role in our analysis. The group has infinite abelianisation if and only if $6|n$ and in these cases $\beta(\Gamma_n(1,n/2-1)^\mathrm{ab})=2$, by Theorem~\ref{thm:betti}. In the remaining cases, $(n,6)=2$, we now calculate the order $|\Gamma_n(1,n/2-1)^\mathrm{ab}|$. One motivation for this is to search for invariants to distinguish $\Gamma_n(1,n/2-1)$ from other groups $\Gamma_n(k,l)$. Let $L_n$ denote the Lucas number of order~$n$, i.e.\,$L_0=2,L_1=1,L_{j+2}=L_{j+1}+L_j$ ($j\geq 0$). Theorem~\ref{thm:Gamman1n/2-1finiteabelianisation} gives the order $|\Gamma_n(1,n/2-1)^\mathrm{ab}|$ in terms of the Lucas numbers, when it is finite. The formula presented has an attractive similarity to the order of the abelianisation of the Fibonacci groups $F(2,n)=G_n(1,2)$, which is given by $|F(2,n)^\mathrm{ab}|=L_n-1-(-1)^n$ (see~\cite{Conwayetal}). In particular if $(n,12)=2$ then $|\Gamma_n(1,n/2-1)^\mathrm{ab}|=3|F(2,n/2)^\mathrm{ab}|$, suggesting a deeper relationship between the groups $\Gamma_n(1,n/2-1)$ and $F(2,n/2)$.

\begin{theorem}\label{thm:Gamman1n/2-1finiteabelianisation}
Let $(n,6)=2$. Then the order $|\Gamma_n(1,n/2-1)^\mathrm{ab}|=3(L_{n/2}-1-(-1)^{n/2})$.
\end{theorem}

\begin{proof}
Let $n=2m$ where $(m,3)=1$. By~(\ref{eq:Gaborder}) we have 
\begin{alignat}{1}
  |\Gamma_n(1,n/2-1)^\mathrm{ab}|
&=\big \vert \prod_{j=0}^{2m-1} 1+\zeta_{2m}^j+\zeta_{2m}^{(m-1)j}\big \vert \nonumber\\
&=\big \vert \prod_{\alpha=0}^{m-1} 1+\zeta_{2m}^{2\alpha}+\zeta_{2m}^{(m-1)\cdot 2\alpha} \big \vert\nonumber\\
&\ \qquad \qquad  \cdot  \big \vert \prod_{\alpha=0}^{m-1} 1+\zeta_{2m}^{2\alpha+1}+\zeta_{2m}^{(m-1)\cdot (2\alpha+1)}\big \vert .\label{eq:P}
\end{alignat}
Now
\begin{alignat}{1}
   \prod_{\alpha=0}^{m-1} 1+\zeta_{2m}^{2\alpha}+\zeta_{2m}^{(m-1)\cdot 2\alpha}
&= 3 \prod_{\alpha=1}^{m-1} 1+\zeta_{m}^{\alpha}+\zeta_{m}^{-\alpha}\nonumber\\
&= 3 \prod_{\alpha=1}^{m-1} \zeta_m^{-\alpha} \cdot \prod_{\alpha=1}^{m-1} \zeta_{m}^{2\alpha}+\zeta_{m}^{\alpha}+1\nonumber\\
&= 3 (\pm 1) \prod_{\alpha=1}^{m-1} \left( \frac{\zeta_{m}^{3\alpha}-1}{\zeta_{m}^{\alpha}-1} \right)\nonumber\\
&= \pm 3 \frac{\prod_{\alpha=1}^{m-1} \zeta_{m}^{3\alpha}-1}{\prod_{\alpha=1}^{m-1} \zeta_{m}^{\alpha}-1}\nonumber\\
&= \pm 3 \frac{\prod_{\alpha=1}^{m-1} {\zeta_m^{\alpha}-1}}{\prod_{\alpha=1}^{m-1}
{\zeta_m^\alpha-1}}\quad \mathrm{since}~(m,3)=1\nonumber\\
&=\pm 3.\label{eq:Peven}
\end{alignat}
Further,
\begin{alignat}{1}
   \prod_{\alpha=0}^{m-1} 1+\zeta_{2m}^{2\alpha+1}+\zeta_{2m}^{(m-1)\cdot (2\alpha+1)}
&= \prod_{\alpha=0}^{m-1} 1+\zeta_{2m}^{2\alpha+1}+\zeta_{2m}^{m(2\alpha+1)}\cdot \zeta_{2m}^{-(2\alpha+1)}\nonumber\\
&= \prod_{\alpha=0}^{m-1} 1+\zeta_{2m}^{2\alpha+1}+\zeta_{2}\cdot \zeta_{2m}^{-(2\alpha+1)}\nonumber\\
&= \prod_{\alpha=0}^{m-1} 1+\zeta_{2m}^{2\alpha+1}-\zeta_{2m}^{-(2\alpha+1)}\nonumber\\
&= \prod_{\alpha=0}^{m-1} 1+2i\sin ((2\alpha+1)/2m)\nonumber\\
&= 1-L_m+(-1)^m\label{eq:Podd}
\end{alignat}
by equation~(10) of~\cite{GarnierRamare09}, and the comment after Corollary~1 of~\cite{GarnierRamare09}. Equations~(\ref{eq:P}),(\ref{eq:Peven}),(\ref{eq:Podd}) combine to complete the proof.
\end{proof}

\begin{corollary}\label{cor:Gamman1n/2-1pairwisenonisom}
Suppose $n\neq n'$ and either $(n,6)=2$ or $(n',6)=2$. Then $\Gamma_n(1,n/2-1)\not \cong \Gamma_{n'}(1,n'/2-1)$.
\end{corollary}

\begin{proof}
If $\Gamma_n(1,n/2-1)$ and $\Gamma_{n'}(1,n'/2-1)$ are isomorphic then both groups have the same abelianisation so $(n,6)=(n',6)=2$ and the orders of these abelianisations, given by Theorem~\ref{thm:Gamman1n/2-1finiteabelianisation}, are equal if and only if $n=n'$.
\end{proof}

We now conjecture the minimum number of generators of the abelianisation $\Gamma_n(1,n/2-1)^\mathrm{ab}$ when this is finite. (This has been verified using GAP for $n\leq 500$).

\begin{conjecture}\label{conj:dGamman}
Suppose $(n,6)=2$, $n\geq 8$. Then
\[d(\Gamma_n(1,n/2-1)^\mathrm{ab})=\begin{cases}
  1& \mathrm{if}~(n,16)=2,\\
  2& \mathrm{if}~(n,16)=4~\mathrm{or}~16,\\
  3& \mathrm{if}~(n,16)=8.\\
\end{cases}\]
\end{conjecture}

It would follow from this and from Theorem~\ref{thm:Gamman1n/2-1finiteabelianisation}, the comments that precede it, and the fact that $L_m$ is a multiple of 3 if and only if $m\equiv 2$~mod~$4$, that if $(n,12)=2$ then
$\Gamma_n(1,n/2-1)^\mathrm{ab}\cong \Z_3\oplus F(2,n/2)^\mathrm{ab}\cong\Z_{3L_{n/2}}$. In further support of the third case, by Corollary~\ref{cor:Gamman1n/2-1n=8mod16}(b) we have $d(\Gamma_n(1,n/2-1)^\mathrm{ab})\geq 3$ when $(n,16)=8$.
When $16|n$ $n=2^tm$ where $m$ is odd and $t\geq 4$. In these cases $\Gamma_n(1,n/2-1)$ maps onto $\Gamma_{2^t}(1,2^{t-1}-1)$ and we further conjecture that for $t\geq 3$ we have $\Gamma_{2^t}(1,2^{t-1}-1)^\mathrm{ab}\cong \Z_3 \oplus \Z_{L_{2^{t-2}}}^2$, where $L_j$ denotes the $j$'th Lucas number (we have verified this for $3\leq t\leq 12$). If that is indeed the case then $d(\Gamma_n(1,n/2-1)^\mathrm{ab})\geq 2$ in the second part of line 2 of Conjecture~\ref{conj:dGamman}.

\section{Isomorphism classes when $n$ has few prime factors}\label{sec:n=p^aq^b}

We first note the following:

\begin{lemma}\label{lem:(n,k)=1or(n,l)=1or(n,k-l)=1}
Let $n\geq 3$, $1\leq k,l <n$, $k\neq l$. If $(n,k)=1$ or $(n,l)=1$ or $(n,k-l)=1$ then $\Gamma_n(k,l)\cong \Gamma_n(1,l')$ for some $1\leq l'< n$.
\end{lemma}

\begin{proof}
By Lemma~\ref{lem:gammaiso}(ii),(iv) $\Gamma_n(k,l)\cong\Gamma_n(l,k)\cong \Gamma_n(k-l,k)$ so the result follows from Lemma~\ref{lem:gammaiso}(v).
\end{proof}

In many cases we are therefore reduced to considering the groups $\Gamma_n(1,l)$. We note some isomorphisms within this class of groups.

\begin{lemma}\label{lem:gamman1liso}
\begin{itemize}
  \item[(a)] If $(n,l)=1$ then $\Gamma_n(1,l)\cong \Gamma_n(1,l')$ where $ll'\equiv 1$~mod~$n$;
  \item[(b)] if $(n,l)=1$ then $\Gamma_n(1,l)\cong \Gamma_n(1,n+1-l')$ where $ll'\equiv 1$~mod~$n$;
  \item[(c)] if $(n,l-1)=1$ then $\Gamma_n(1,l)\cong \Gamma_n(1,1+\bar{l})$ where $(l-1)\bar{l}\equiv 1$~mod~$n$;
  \item[(d)] if $(n,l-1)=1$ then $\Gamma_n(1,l)\cong \Gamma_n(1,n-\bar{l})$ where $(l-1)\bar{l}=1$~mod~$n$.
\end{itemize}
\end{lemma}

\begin{proof}
(a),(b) We have $\Gamma_n(1,l)\cong \Gamma_n(l,1) \cong \Gamma_n(1,l') \cong \Gamma_n(1,n+1-l')$ by applying parts (ii),(v),(iv) of Lemma~\ref{lem:gammaiso} in turn.

(c) We have $\Gamma_n(1,l) \cong \Gamma_n(l,1) \cong \Gamma_n(l-1,-1) \cong \Gamma_n(1,-\bar{l}) \cong \Gamma_n(1,1+\bar{l})$ by applying parts (ii),(iii),(v),(iv) of Lemma~\ref{lem:gammaiso} in turn.

(d) Putting $m=1-l'$, $\bar{m}=n-l$ in part (b) gives that if $(n,\bar{m})=1$ then $\Gamma_n(1,n-\bar{m})\cong \Gamma_n(1,m)$ where $(m-1)\bar{m}\equiv 1$~mod~$n$ which, after relabelling, is the statement of part~(d).
\end{proof}

\begin{lemma}\label{lem:Gamman1lLsets}
Let $n\geq 3$, $1< l< n$.
\begin{itemize}
\item[(a)] If $n\equiv 0$~mod~$4$ and $n\geq 4$ then $\Gamma_n(1,l)\cong \Gamma_n(1,l')$ for some $l'\in L$, where $L=\{2,\ldots , n/2\}$;
\item[(b)] if $n\equiv 2$~mod~$4$ and $n\geq 10$ then $\Gamma_n(1,l)\cong \Gamma_n(1,l')$ for some $l'\in L$, where $L=\{2,\ldots , n/2\}\backslash \{(n+2)/4\}$;
\item[(c)] if $n\equiv 3$~mod~$6$ and $n\geq 9$ then $\Gamma_n(1,l)\cong \Gamma_n(1,l')$ for some $l'\in L$, where $L=\{2,\ldots ,(n-3)/2\}$;
\item[(d)] if $n\equiv 1$~mod~$6$ and $n\geq 9$ then $\Gamma_n(1,l)\cong \Gamma_n(1,l')$ for some $l'\in L$, where $L=\{2,\ldots ,(n-3)/2\}\backslash \{(n+2)/3\}$;
\item[(e)] if $n\equiv 5$~mod~$6$ and $n\geq 9$ then $\Gamma_n(1,l)\cong \Gamma_n(1,l')$ for some $l'\in L$, where $L=\{2,\ldots ,(n-3)/2\}\backslash \{(n+1)/3\}$.
\end{itemize}
\end{lemma}

\begin{proof}
If $l>n/2$ then $l'=n+1-l\in\{2,\ldots ,n/2\}$ and by Lemma~\ref{lem:gammaiso}(iv) we have $\Gamma_n(1,l)\cong \Gamma_n(1,l')$. Thus we may assume $L\subseteq \{2,\ldots ,n/2\}$. In particular, part (a) is proved.
To prove (b) we show that $\Gamma_n(1,(n+2)/4)\cong \Gamma_n(1,n/2-1)$. Let $n\equiv 2\epsilon$~mod~$8$ ($\epsilon=\pm 1$). Now $((n+2\epsilon)/4,n)=1$ and $(n+2\epsilon)/4 \cdot (n+4\epsilon)/2\equiv 1$~mod~$n$ so the result follows from parts (b),(c) of Lemma~\ref{lem:gamman1liso} for $\epsilon=1$ and $\epsilon=-1$, respectively.

Now suppose that $n$ is odd. Then by applying parts (v) (with $\psi=2$),(iii),(ii) of Lemma~\ref{lem:gammaiso} in turn, we have $\Gamma_n(1,(n-1)/2)\cong \Gamma_n(1,3)$, so we may assume $L\subseteq \{2,\ldots ,(n-3)/2\}$. In particular, part (c) is proved. If $n\equiv 1$~mod~$6$ then by applying parts (v) (with $\psi=3$),(iv),(ii) of Lemma~\ref{lem:gammaiso} in turn we have $\Gamma_n(1,(n+2)/3)\cong \Gamma_n(1,3)$, proving (d). If $n\equiv 5$~mod~$6$ then by applying parts (v) (with $\psi=3$),(ii) of Lemma~\ref{lem:gammaiso} in turn we have $\Gamma_n(1,(n+1)/3)\cong \Gamma_n(1,3)$, proving (e).
\end{proof}

We apply this to study isomorphism classes and the set $S(n)$ defined in the Introduction when either $n$ has at most two prime factors or when (A) holds and $n$ has at most three prime factors.

\begin{corollary}\label{cor:Gammap^aq^bLsets}
\begin{itemize}
  \item[(a)] Suppose $n=p^\alpha q^\beta$, where $p,q$ are distinct primes, $\alpha,\beta\geq 0$, $1\leq k,l<n$, $k\neq l$, and $(n,k,l)=1$. Then $\Gamma_n(k,l)\cong \Gamma_n(1,l')$ for some $l'\in L$ where $L$ is as given in Lemma~\ref{lem:Gamman1lLsets}, and hence
$S(n)=\{\Gamma_n(1,l)\ |\ l\in L\}$.
\item[(b)] Suppose $n=p^\alpha q^\beta r^\gamma$, where $p,q,r$ are distinct primes, $\alpha,\beta,\gamma\geq 0$, $1\leq k,l<n$, $k\neq l$, $(n,k,l)=1$, and that (A) holds. Then $\Gamma_n(k,l)\cong \Gamma_n(1,l')$ for some $l'\in L$ where $L$ is as given in Lemma~\ref{lem:Gamman1lLsets}.
\end{itemize}
\end{corollary}

\begin{proof}
(a) At least one of the conditions $(n,k)=1$, $(n,l)=1$, or $(n,k-l)=1$ holds so by Lemma~\ref{lem:(n,k)=1or(n,l)=1or(n,k-l)=1} $\Gamma_n(k,l)\cong \Gamma_n(1,l')$ for some $1<l'<n$ and the result follows from Lemma~\ref{lem:Gamman1lLsets}.

(b) Since (A) holds we have (without loss of generality) that $p=3$ and that $3$ does not divide $k$ or $l$. Suppose for contradiction that $(n,k)>1$ and $(n,l)>1$ and $(n,k-l)>1$. Then (without loss of generality) $q$ divides $k$. Since $(n,k,l)=1$ we have that $q$ does not divide $l$, so $r$ divides $l$ and hence $r$ does not divide $k$. Then $(n,k-l)>1$ implies that $p$ divides $k-l$, but since $p$ divides $k+l$ it also divides $2k$, a contradiction. Therefore $(n,k)=1$ or $(n,l)=1$ or $(n,k-l)=1$ so by Lemma~\ref{lem:(n,k)=1or(n,l)=1or(n,k-l)=1} $\Gamma_n(k,l)\cong \Gamma_n(1,l')$ for some $1<l'<n$ and the result follows from Lemma~\ref{lem:Gamman1lLsets}.
\end{proof}

The hypothesis that (A) holds cannot be directly removed from Corollary~\ref{cor:Gammap^aq^bLsets}(b). To see this, it suffices to observe that $\Gamma_{30}(2,5)\not \cong \Gamma_{30}(1,l')$ for any $l'$ (which can be seen by comparing abelianisations). Nor can Corollary~\ref{cor:Gammap^aq^bLsets}(b) be extended to the case when $n$ has at least 4 prime factors, since $\Gamma_{210}(2,7)\not \cong \Gamma_{210}(1,l')$ for any $l'$ (again by comparing abelianisations). On the other hand, if we assume that $p\geq 3$ and $\alpha\geq 1$ then the hypothesis that (A) holds can be replaced by $k+l\equiv 0$~mod~$p$, the proof proceeding as before.

Recalling that if any of $k,l,k-l$ is coprime to $n$ then $\Gamma_n(k,l)$ is isomorphic to some $\Gamma_n(1,l')$, we now turn to the case $n=pqr$ where $p,q,r$ are distinct primes. For $p=2$ there is at most one group $\Gamma_n(k,l)$ that is not isomorphic to any $\Gamma_n(1,l')$:

\begin{lemma}\label{lem:2qr}
Let $n=2qr$ where $q,r\geq 3$ are distinct primes and suppose $1\leq k,l<n$ satisfy $(n,k,l)=1$, $(n,k)>1$, $(n,l)>1$ and $(n,k-l)>1$. Then $\Gamma_n(k,l)\cong \Gamma_n(q,r)$.
\end{lemma}

\begin{proof}
Since $(n,k,l)=1$, it follows that at most one of $k,l$ is even, and since $\Gamma_n(k,l)\cong \Gamma_n(k,k-l)\cong\Gamma_n(l,l-k)$ we may further assume that both $k,l$ are odd and since $\Gamma_n(k,l)\cong \Gamma_n(l,k)$ we may assume that $q|k$, and $r|l$. If then $q|l$ or $r|k$ we get a contradiction to $(n,k,l)=1$ so $(n,k)=q$ and $(n,l)=r$. Let
\[ \nu = (k^{-1}~\mathrm{mod}~r)q+(l^{-1}~\mathrm{mod}~q)r\]
and let $\psi= \nu$, if $\nu$ is odd, and let $\psi= \nu+qr$, if $\nu$ is even. Then $(\nu,q)=1,(\nu,r)=1$ so $(\nu,qr)=1$ and hence (since $\psi$ is odd), $(\psi,2qr)=1$. Now $\nu (k/q)\equiv 1$~mod~$r$ and $\nu (l/r)\equiv 1$~mod~$q$ so $\psi (k/q)\equiv 1$~mod~$r$, $\psi (l/r)\equiv 1$~mod~$q$. Thus  $\psi (k/q)\equiv 1$ or $r+1$~mod~$2r$, $\psi (l/r)\equiv 1$ or $q+1$~mod~$2q$. But $\psi, k/q,l/r$ are all odd, so $\psi (k/q)\equiv 1$~mod~$2r$, $\psi (l/r)\equiv 1$~mod~$2q$, and so $\psi k\equiv q$~mod~$2qr$, $\psi l\equiv r$~mod~$2qr$. Then, by Lemma~\ref{lem:gammaiso}(v), we have $\Gamma_n(k,l)\cong \Gamma_n(\psi k,\psi l)=\Gamma_n(q,r)$, as required.
\end{proof}

For $p\geq 5$ we now conjecture the value of $f(p)$, which we have verified for all $p\leq 293$.

\begin{conjecture}\label{conj:f(p)}
Let $p\geq 5$ be prime. Then
\[f(p)=\begin{cases}
  (p+5)/6& \mathrm{if}~p\equiv 1~mod~6,\\
  (p+1)/6& \mathrm{if}~p\equiv 5~mod~6.
\end{cases}\]
\end{conjecture}

\section{The groups $\Gamma_n(k,l)$ for $n\leq 29$}\label{sec:n<19}

As seen in Section~\ref{sec:n=p^aq^b} if $n$ has at most two prime factors then $\Gamma_n(k,l)\cong\Gamma_n(1,l')$ for some $l'\in L$, where $L$ is as given in Lemma~\ref{lem:Gamman1lLsets}. Since the least integer that has more than two prime factors is 30, in this section we provide a complete list of the groups $\Gamma_n(k,l)$ (up to isomorphism) for $n\leq 29$.

Using Lemma~\ref{lem:gamman1liso}, the following additional isomorphisms can be found amongst the groups $\Gamma_n(1,l')$, $l'\in L$, where we write $(n;l_1,\ldots,l_t)$ to mean $\Gamma_n(1,l_1)\cong \cdots \cong \Gamma_n(1,l_t)$:
\begin{alignat*}{1}
&\ (14;3,5), (16;3,6), (16;4,5), (17;4,5,7), (18;6,7), (19;4,5,6), (20;3,7), \\
&\  (20;4,8), (22;3,8), (22;4,7), (22;5,9), (23;4,6,9), (23;5,7,10),\\
&\ (25;4,7,8), (25;5,6), (25;10,11), (26;3,9), (26;4,10), (26;5,6), (26;8,11),\\
&\ (27;4,7), (27;5,8,11), (27;6,12), (27;9,10), (28;3,10), (28;4,9), (28;5,12),\\
&\ (28;6,11), (29;4,8,11), (29;5,6,7), (29;9,12,13).
\end{alignat*}

(For example, $(16;3,6)$ follows from Lemma~\ref{lem:gamman1liso}(b) by setting $n=16,l=3,l'=11$.)

Thus, for any given $3\leq n\leq 29$ we obtain a list $L$ of values $l'$ such that any given group $\Gamma_n(k,l)$ is isomorphic to $\Gamma_n(1,l')$ for some $l'\in L$. We wish to show that these lists are the smallest possible, in the sense that $\Gamma_n(1,l)\not  \cong \Gamma_n(1,l')$ for distinct $l,l'\in L$. In most cases we are able to show this by considering the groups' structural properties given by Table~\ref{tab:EdjvetWilliams} (whether the groups or their abelianisations are finite or infinite, the groups' deficiencies, or the structure of the abelianisations) or, when this is not enough, by computing and comparing abelianisations (which is easy to do using GAP). In the few cases where we are unable to distinguish groups this way, we can prove non-isomorphism by comparing the second derived quotients. (Specifically, these are the cases $\Gamma_{10}(1,4)\not\cong\Gamma_{10}(1,5)$, $\Gamma_{16}(1,4)\not\cong\Gamma_{16}(1,8)$, $\Gamma_{18}(1,4)\not\cong\Gamma_{18}(1,9)$, $\Gamma_{20}(1,4)\not\cong \Gamma_{20}(1,5)$, $\Gamma_{22}(1,4)\not\cong \Gamma_{22}(1,5)$.) In this way we are able to determine precisely the sets $S(n)$ for $3\leq n\leq 29$. For $6\leq n\leq 29$ the groups in these sets are listed in Table~\ref{tab:CRSgroupsn<30}, where (for a given $n$) the groups are pairwise non-isomorphic, and so the value of $f(n)$ can be immediately obtained as the number of values of $l$ listed. (For $n\leq 5$ the group $\Gamma_n(k,l)\cong \Gamma_n(1,2)$ which is $\Z*\Z$ for $n=3$ and is $\Z_3$ for $n=4,5$.)

The identification of the group as $\Z*\Z$, $\Z*\Z*\Z_\gamma$, `Metacyclic' (referring to the group $B((2^n-(-1)^n)/3,3,2^{2n/3},1)$), or as an infinite group can be obtained from Table~\ref{tab:EdjvetWilliams}, and the identification of the presentation as C(3)-T(6)-special or non-special (denoted `C(3)-T(6)-s' and `C(3)-T(6)-ns') follow from Theorems~\ref{thm:C3T6} and~\ref{thm:C3T6special}. The groups $\Gamma_7(1,3), \Gamma_8(1,3)$ and $\Gamma_{20}(1,9)$ were discussed in Examples~\ref{ex:G_7(1,3)},\ref{ex:G_8(1,3)},  and~\ref{ex:G_20(1,9)}, respectively.

\begin{table}
\begin{center}
\begin{tabular}{@{} l l}
\begin{tabular}{|@{\,}p{0.4cm}@{\,}|p{1.7cm}|p{3.0cm}|}
\hline
$n$ & $l$ &$\Gamma_n(1,l)$\\
\hline
$6$ & $2$ & $\Z*\Z$\\
    & $3$ & metacyclic\\\hline
$7$ & $2$ & $\Z_3$\\
    & $3$ & C(3)-T(6)-s\\\hline
$8$ & $2$ & $\Z_3$\\
    & $3$ & large\\
    & $4$ & C(3)-T(6)-ns\\\hline
$9$ & $2$ & $\Z*\Z$\\
    & $3$ & metacyclic\\\hline
$10$ & $2$ & $\Z_3$\\
     & $4$ & infinite\\
     & $5$ & C(3)-T(6)-ns\\\hline
$11$ & $2$ & $\Z_3$\\
     & $3$ & C(3)-T(6)-ns\\\hline
$12$ & $2$ & $\Z* \Z$\\
     & $3,6$ & C(3)-T(6)-ns\\
     & $4$ & metacyclic\\
     & $5$ & $\Z_5*\Z*\Z$\\\hline
$13$ & $2$ & $\Z_3$\\
     & $3,4$ & C(3)-T(6)-ns\\\hline
$14$ & $2$ & $\Z_3$\\
     & $3,7$ & C(3)-T(6)-ns\\
     & $6$ & infinite\\\hline
$15$ & $2$ & $\Z* \Z$\\
     & $3,4$ & C(3)-T(6)-ns\\
     & $5$ & $\Z_{11}*\Z*\Z$\\
     & $6$ & metacyclic\\\hline
$16$ & $2$ & $\Z_3$\\
     & $3,4,8$ & C(3)-T(6)-ns\\
     & $7$ & infinite\\\hline
$17$ & $2$ & $\Z_3$\\
     & $3,4$ & C(3)-T(6)-ns\\\hline
$18$ & $2$ & $\Z* \Z$\\
     & $3,4,9$ & C(3)-T(6)-ns\\
     & $8$ & $\Z_{19}*\Z*\Z$\\
     & $6$ & metacyclic\\\hline
\end{tabular}

&

\begin{tabular}{|@{\,}p{0.6cm}@{\,}|p{1.7cm}|p{3.2cm}|}
\hline
$n$ & $l$ &$\Gamma_n(1,l)$\\
\hline
$19$ & $2$ & $\Z_3$\\
     & $3,4,8$ & C(3)-T(6)-ns\\\hline
$20$ & $2$ & $\Z_3$\\
     & $3-6,10$ & C(3)-T(6)-ns\\
     & $9$ & large\\\hline
$21$ & $2$ & $\Z*\Z$\\
     & $3,4,6,9$ & C(3)-T(6)-ns\\
     & $5$ & C(3)-T(6)-ns, large\\
     & $7$ & metacyclic\\
     & $8$ & $\Z_{43}*\Z*\Z$\\\hline
$22$ & $2$ & $\Z_3$\\
     & $3-5,11$ & C(3)-T(6)-ns\\
     & $10$ & infinite\\\hline
$23$ & $2$ & $\Z_3$\\
     & $3-5$ & C(3)-T(6)-ns\\\hline
$24$ & $2$ & $\Z*\Z$\\
     & $3,4,6,$ & C(3)-T(6)-ns\\
     & $7,10,12$ & \\
     & $5$ & C(3)-T(6)-ns, large\\
     & $8$ & $\Z_{85}*\Z*\Z$\\
     & $9$ & metacyclic\\
     & $11$ & large\\\hline
$25$ & $2$ & $\Z_3$\\
     & $3-5,10$ & C(3)-T(6)-ns\\\hline
$26$ & $2$ & $\Z_3$\\
     & $3-5,8,13$ & C(3)-T(6)-ns\\
     & $12$ & infinite\\\hline
$27$ & $2$ & $\Z*\Z$\\
     & $3,4,6$ & C(3)-T(6)-ns\\
     & $5$ & C(3)-T(6)-ns, large\\
     & $9$ & metacyclic\\\hline
$28$ & $2$ & $\Z_3$\\
     & $3-8,14$ & C(3)-T(6)-ns\\
     & $13$ & infinite\\\hline
$29$ & $2$ & $\Z_3$\\
     & $3-5,9$ & C(3)-T(6)-ns\\\hline
\end{tabular}
\end{tabular}
\end{center}

  \caption{The groups $\Gamma_n(1,l)$ for $6\leq n\leq 29$.\label{tab:CRSgroupsn<30}}
\end{table}

\section{The sets $S(n)$ for $n\geq 19$}\label{sec:ABCD}

When $n\geq 19$ the structural properties given in Table~\ref{tab:EdjvetWilliams} are determined purely by the (A),(B),(C) conditions, and not by the value of $n$ (and this is not the case for $n=18$). In this section we analyse the groups $\Gamma_n(k,l)$ for $n\geq 19$, the cases $n\leq 18$ being covered by the results of Section~\ref{sec:n<19}.

Let $V=\{(a,b,c,d)\ |\ a,b,c,d\in \{T,F\}\}$. We shall write elements of $V$ either as $(a,b,c,d)$ or as $abcd$. For $v=(a,b,c,d)\in V$, let
\begin{alignat*}{1}
S^v(n)
&=S^{(a,b,c,d)}(n)\\
&=\onetwoset{\Gamma_n(k,l)}{1\leq k,l< n, k\not=l, (n,k,l)=1, A(n,k,l)=a,}{B(n,k,l)=b, C(n,k,l)=c, D(n,k,l)=d}
\end{alignat*}
(where isomorphic groups are identified) and let $f^{(a,b,c,d)}(n)=|S^{(a,b,c,d)}(n)|$. For example
\[S^{FFTF}(6)=\twolineset{\Gamma_6(1,3), \Gamma_6(1,4), \Gamma_6(2,3), \Gamma_6(2,5), \Gamma_6(3,1),\Gamma_6(3,2),}{ \Gamma_6(3,4), \Gamma_6(3,5), \Gamma_6(4,1), \Gamma_6(4,3), \Gamma_6(5,2),\Gamma_6(5,3)}\]
but each of these groups are isomorphic to $\Z_9$, so $S^{FFTF}(6)=\{\Z_9\}$ and $f^{FFTF}(6)=1$. Thus $S(n)=\bigcup_{v\in V} S^v(n)$, but note that it is possible to have $S^v(n)\cap S^{v'}(n)\neq \emptyset$ for $v,v'\in V, v \neq v'$. For instance $\Gamma_6(1,2)\cong \Z*\Z$ and $\Gamma_6(1,5)\cong\Z*\Z$ so $\Z*\Z \in S^{TFTT}(6)\cap S^{TTTT}(6)$. We believe that for $n>6$ if $v\neq v'$ then $S^v(n)\cap S^{v'}(n)= \emptyset$ but we have been unable to prove this -- see Proposition~\ref{prop:nonintersection} and Conjecture~\ref{conj:possibleisoms}.

The following theorem gives an expression for $S(n)$ according to the value of $(n,18)$ and reduces the problem of determining the elements of the sets $S(n)$ to that of determining the elements of the sets $S^{FFFF}(n)$ and $S^{TFFF}(n)$. Similarly, it reduces the problem of calculating the value of $f(n)$ to that of calculating $f^{FFFF}(n)$ and $f^{TFFF}(n)$. (For example, part~(a) implies that for $(n,18)=1$ $f(n)=1+f^{FFFF}(n)$.) We use the symbol $\dot{\cup}$ to denote disjoint union (i.e.\,writing $A\dot{\cup} B$ implies that $A\cap B=\emptyset$).

\begin{theorem}\label{thm:ABCDsets}
Let $n\geq 19$ and, when $3|n$, let $\gamma=(2^{n/3}-(-1)^{n/3})/3$ and let $B$ be the metacyclic group $B((2^n-(-1)^n)/3,3,2^{2n/3},1)$.
\begin{itemize}
  \item[(a)] If $(n,18)=1$ then 
  \(S(n)=\{\Z_3\} \dot{\cup}  S^{FFFF}(n).\)
  \item[(b)] Suppose $(n,18)=2$.
  \begin{itemize}
    \item [(i)] If $\Gamma_n(1,n/2-1)\not \in S^{FFFF}(n)$, then
    \[S(n)=\{\Z_3\} \dot{\cup} \{\Gamma_n(1,n/2-1)\} \dot{\cup}\,S^{FFFF}(n);\]
    \item [(ii)] if $\Gamma_n(1,n/2-1) \in S^{FFFF}(n)$, then
    \[S(n)=\{\Z_3\} \dot{\cup} S^{FFFF}(n).\]
  \end{itemize}
  \item[(c)] If $(n,18)=3$ then
  \[ S(n)=\{\Z*\Z\} \dot{\cup} \{B\} \dot{\cup}\{\Z*\Z*\Z_{\gamma}\} \dot{\cup}  S^{TFFF}(n)\dot{\cup} S^{FFFF}(n).\]
  \item [(d)] Suppose $(n,18)=6$.
  \begin{itemize}
    \item[(i)] If $\Gamma_n(1,n/2-1)\not \in S^{TFFF}(n)$, then
    \begin{alignat*}{1}
    S(n)&=\{\Z*\Z\} \dot{\cup} \{B\} \dot{\cup} \{\Z*\Z*\Z_\gamma \} \dot{\cup} \{\Gamma_n(1,n/2-1)\}\\
        &\qquad \dot{\cup}  S^{TFFF}(n)\dot{\cup} S^{FFFF}(n);
    \end{alignat*}
    \item[(ii)] if $\Gamma_n(1,n/2-1) \in S^{TFFF}(n)$, then
  \[ S(n)=\{\Z*\Z\} \dot{\cup} \{B\} \dot{\cup} \{\Z*\Z*\Z_\gamma\} \dot{\cup}  S^{TFFF}(n)\dot{\cup} S^{FFFF}(n).\]
  \end{itemize}

  \item[(e)] If $(n,18)=9$ then
  \[ S(n)=\{\Z*\Z\} \dot{\cup} \{B\} \dot{\cup}  S^{TFFF}(n)\dot{\cup} S^{FFFF}(n).\]
  \item[(f)] Suppose $(n,18)=18$.
  \begin{itemize}
    \item[(i)] If $\Gamma_n(1,n/2-1)\not \in S^{TFFF}(n)$, then
 \[S(n)=\{\Z*\Z\} \dot{\cup} \{B\} \dot{\cup} \{\Gamma_n(1,n/2-1)\}\dot{\cup}  S^{TFFF}(n)\dot{\cup} S^{FFFF}(n);\]
    \item[(ii)] if $\Gamma_n(1,n/2-1)\in S^{TFFF}(n)$, then
 \[S(n)=\{\Z*\Z\} \dot{\cup} \{B\} \dot{\cup} S^{TFFF}(n)\dot{\cup} S^{FFFF}(n).\]
  \end{itemize}
\end{itemize}
\end{theorem}

We prove Theorem~\ref{thm:ABCDsets} through a series of propositions that consider the sets $S^v(n)$ for each
\(v\in V=\{(a,b,c,d)\ |\ a,b,c,d\in \{T,F\}\}\).

\begin{proposition}\label{prop:eightemptyset}
Let $n\geq 7$. If either $v\in \{(F,T,T,F),(T,T,T,F)$,  \linebreak $(T,T,F,F), (F,T,F,F), (T,T,T,T), (F,F,T,T), (F,T,T,T)\}$ or $n\neq 12$ and $v=(T,F,T,T)$ then $S^v(n)=\emptyset$.
\end{proposition}

\begin{proof}
  If condition (B) holds then (D) also holds, so $S^{aTbF}(n)=\emptyset$ for all $n$, $a,b\in \{T,F\}$.
  If (C) and (D) hold then $6|n$ and $n/6$ divides $(n,k,l)$ when $n\equiv 2$~mod~$4$ and $n/12$ divides $(n,k,l)$ when $n\equiv 0$~mod~$4$. The definition of $S^v(n)$ requires $(n,k,l)=1$ so $n=6$ or $12$. We have $n\neq 6$ by hypothesis, and if $n=12$ it is easy to check that $S^{FFTT}(n)=S^{FTTT}(n)=\emptyset$, and the proof is complete.
\end{proof}

\begin{proposition}\label{prop:TTFT}
Let $n\geq 9$. Then
\[
 S^{TTFT}(n)=
 \begin{cases}
   \{\Z*\Z\} &\mathrm{if}~n\equiv 0~\mathrm{mod}~3,\\
   \emptyset & \mathrm{otherwise}.
 \end{cases}
\]
\end{proposition}

\begin{proof}
If $n\not \equiv 0$~mod~$3$ then (A) does not hold so $S^{TTFT}=\emptyset$ so suppose $n\equiv 0$~mod~$3$. By~\cite[Lemma~2.4]{EdjvetWilliams} if $\Gamma_n(k,l)\in S^{TTFT}$ then  $\Gamma_n(k,l)\cong \Z*\Z$ and since $\Gamma_n(1,2)\in S^{TTFT}$ we have $S^{TTFT}=\{\Z*\Z\}$.
\end{proof}

\begin{proposition}\label{prop:FTFT}
Let $n\geq 3$. Then
\[
 S^{FTFT}(n)=
 \begin{cases}
   \emptyset &\mathrm{if}~n\equiv 0~\mathrm{mod}~3,\\
   \{\Z_3\} & \mathrm{otherwise}.
 \end{cases}
\]
\end{proposition}

\begin{proof}
If $n\equiv 0$~mod~$3$ then it follows that if (B) holds then (A) holds, and hence $S^{FTFT}=\emptyset$ so suppose $n\not \equiv 0$~mod~$3$. By~\cite[Lemma~2.4]{EdjvetWilliams} if $\Gamma_n(k,l)\in S^{FTFT}$ then  $\Gamma_n(k,l)\cong \Z_3$ and since $\Gamma_n(1,2)\in S^{FTFT}$ we have $S^{FTFT}=\{\Z_3\}$.
\end{proof}

\begin{proposition}\label{prop:FFTF}
Let $n\geq 6$. Then
\[
 S^{FFTF}(n)=
\begin{cases}
   \{B((2^n-(-1)^n)/3,3,2^{2n/3},1)\} &\mathrm{if}~n\equiv 0~\mathrm{mod}~3,\\
   \emptyset & \mathrm{otherwise}.
\end{cases}
\]
\end{proposition}

\begin{proof}
If $n\not \equiv 0$~mod~$3$ then (C) does not hold so $S^{FFTF}(n)=\emptyset$ so suppose $n\equiv 0$~mod~$3$. If (C) holds, (A) does not hold, and $k\not \equiv 0$, $l\not \equiv 0$, $k\not \equiv l$~mod~$n$ then~\cite[Corollary~D]{BogleyWilliams17} implies that $\Gamma_n(k,l)\cong B((2^n-(-1)^n)/3,3,2^{2n/3},1)$ and since $\Gamma_n(n/3,1+2n/3) \in S^{FFTF}(n)$ we have $S^{FFTF}(n)=\{B((2^n-(-1)^n)/3,3,2^{2n/3},1)\}$.
\end{proof}

\begin{proposition}\label{prop:TFTF}
Let $n\geq 3$ and when $3|n$ let $\gamma = (2^{n/3} - (-1)^{n/3})/3$. Then
\[
 S^{TFTF}(n)=
\begin{cases}
   \{\Z*\Z*\Z_\gamma\} &\mathrm{if}~n\equiv \pm 3~\mathrm{mod}~9,\\
   \emptyset & \mathrm{otherwise}.
\end{cases}
\]
\end{proposition}

\begin{proof}
If (A) holds then $n\equiv 0$~mod~3, so if $S^{TFTF}\neq \emptyset$ we have $n\equiv 0$~mod~$3$. If $n\equiv 3$ or $6$ mod~$9$ then $\Gamma_n(1,n/3+1)\in S^{TFTF}(n)$, and by~\cite[Lemma~2.5]{EdjvetWilliams} $S^{TFTF}(n)=\{\Z*\Z*\Z_\gamma\}$.
Suppose $n\equiv 0$~mod~$9$, $(n,k,l)=1$ and (A) and (C) both hold. Since (A) holds $3$ divides $k$ and $l$, and since (C) holds $n/3$ divides $k$ or $l$ or $k-l$, and hence $3$ divides $k$ or $l$ or $k-l$, a contradiction.
\end{proof}

\begin{proposition}\label{prop:FFFT}
Let $n\geq 3$. Then
\[
 S^{FFFT}(n)=
\begin{cases}
   \{\Gamma_n(1,n/2-1)\} &\mathrm{if}~n\equiv \pm 2~\mathrm{mod}~6,\\
   \emptyset & \mathrm{otherwise}.
\end{cases}
\]
\end{proposition}

\begin{proof}
If $n$ is odd and (D) holds then (B) also holds so $S^{FFFT}(n)=\emptyset$, so assume $n$ is even.
If $n\equiv 0$~mod~$6$ and (D) holds then it follows that (A) holds and hence $S^{FFFT}(n)=\emptyset$, so suppose $n\equiv \pm 2$~mod~$6$.
Lemma~\ref{lem:FFFT=GN(1,n/2-1)} then implies that if (D) holds and (B) does not and $(n,k,l)=1$ and $k\not \equiv 0, l\not \equiv 0, k\not \equiv l$~mod~$n$ then $\Gamma_n(k,l)\cong \Gamma_n(1,n/2-1)$ and since $\Gamma_n(1,n/2-1)\in  S^{FFFT}(n)$ we have $S^{FFFT}(n)=\{\Gamma_n(1,n/2-1)\}$.
\end{proof}

\begin{proposition}\label{prop:TFFT}
Let $n\geq 13$. Then
\[
 S^{TFFT}(n)=
\begin{cases}
   \{\Gamma_n(1,n/2-1)\} &\mathrm{if}~n\equiv  0~\mathrm{mod}~6,\\
   \emptyset & \mathrm{otherwise}.
\end{cases}
\]
\end{proposition}

\begin{proof}
If $n$ is odd then (D) does not hold so $S^{TFFT}(n)=\emptyset$ so assume $n$ is even. If $n\not \equiv 0$~mod~$3$ then (A) does not hold so again $S^{TFFT}(n)=\emptyset$ so assume $n\equiv 0$~mod~$6$. By Lemma~\ref{lem:FFFT=GN(1,n/2-1)} if (D) holds and (B) does not and $(n,k,l)=1$ and $k\not \equiv 0, l\not \equiv 0, k\not \equiv l$~mod~$n$ then $\Gamma_n(k,l)\cong \Gamma_n(1,n/2-1)$ and since $\Gamma_n(1,n/2-1)\in  S^{TFFT}(n)$ we have $S^{TFFT}(n)=\{\Gamma_n(1,n/2-1)\}$.
\end{proof}

Note that Proposition~\ref{prop:TFFT} would be false for $n<13$ since $S^{TFFT}(6)=S^{TFFT}(12)=\emptyset$.

\begin{proposition}\label{prop:elementofTFFF}
Let $n\geq 19$. If $n\equiv 0$~mod~$3$ then $\Gamma_n(1,5)\in S^{TFFF}(n)$ and if $n \not\equiv 0$~mod~$3$ then $S^{TFFF}(n)=\emptyset$.
\end{proposition}

\begin{proof}
Clearly $\Gamma_n(1,5)\in S^{TFFF}(n)$, and if $n \not\equiv 0$~mod~$3$ then (A) does not hold so $S^{TFFF}(n)=\emptyset$.
\end{proof}

In particular, $S^{TFFF}(n)$ is non-empty when $n\equiv 0$~mod~$3$, $n\geq 19$. Note that Proposition~\ref{prop:elementofTFFF} would be false for $n=18$ since $\Gamma_{18}(1,5)\not \in S^{TFFF}(18)$.

\begin{proposition}\label{prop:elementofFFFF}
Let $n\geq 10$. Then $\Gamma_n(1,3)\in S^{FFFF}(n)$ and if $n$ is even then $\Gamma_n(1,n/2)\in S^{FFFF}(n)$.
\end{proposition}

\begin{proof}
Obvious.
\end{proof}

In particular, $S^{FFFF}(n)$ is non-empty when $n\geq 10$. We now turn to the proof of Theorem~\ref{thm:ABCDsets}. All parts proceed in the same way, and so for illustration and brevity we prove only part~(b).

\begin{proof}[Proof of Theorem~\ref{thm:ABCDsets}(b)]
\noindent Suppose $(n,18)=2$.  If $S^v(n)\neq \emptyset$ then Propositions~\ref{prop:eightemptyset}--\ref{prop:elementofFFFF} imply that $v=(F,T,F,T),(F,F,F,T)$ or $(F,F,F,F)$; Proposition~\ref{prop:FTFT} implies that $S^{FTFT}(n)=\{\Z_3\}$ and Proposition~\ref{prop:FFFT} implies that $S^{FFFT}(n)=\{\Gamma_n(1,n/2-1)\}$. With the exception of the pair of sets $S^{FFFT}(n)$ and $S^{FFFF}(n)$, the fact that these sets are pairwise disjoint follows from the invariants given in Table~\ref{tab:EdjvetWilliams}. (Groups in $S^{FFFT}(n)$ and $S^{FFFF}(n)$ are infinite, have finite abelianisation and are of deficiency zero.) Thus $S(n)$ is one of the two sets given in the statement, depending whether or not $\Gamma_n(1,n/2-1)\in S^{FFFF}(n)$.
\end{proof}

Theorem~\ref{thm:ABCDsets}(b),(d),(f) prompts the question as to whether or not, for $n>18$, the group $\Gamma_n(1,n/2-1)$ can be an element of $S^{TFFF}(n)$ (for $n\equiv 0$~mod~6) or of $S^{FFFF}(n)$ (for $n\equiv 2,4$~mod~$6$). Moreover, for $n>6$ we have the following:

\begin{proposition}\label{prop:nonintersection}
Let $v,v'\in V$, $v\neq v'$. If $n>6$ and $S^v(n)\cap S^{v'}(n)\neq \emptyset$ then either:
\begin{itemize}
  \item[(a)] $n\equiv 0$~mod~$6$ and $\Gamma_n(1,n/2-1) \in S^{TFFF}(n)$; or
  \item[(b)] $n\equiv 2,4$~mod~$6$ and $\Gamma_n(1,n/2-1) \in S^{FFFF}(n)$.
\end{itemize}
\end{proposition}

\begin{proof}
Let $v=(a,b,c,d),v'=(a',b',c',d')\in V$ and suppose $S^v(n)\cap S^{v'}(n)\neq \emptyset$. If $(a,b,c)\neq(a',b',c')$ then it follows from the invariants in Table~\ref{tab:EdjvetWilliams} that $S^v(n)\cap S^{v'}(n)=\emptyset$. Thus we may assume $a=a',b=b',c=c',d=T,d'=F$. Then, if $b=T$ we have $S^{v'}(n)=\emptyset$ so we may also assume that $b=F$.

Suppose $c=T$. If $a=F$ then Proposition~\ref{prop:eightemptyset} implies that $S^{v}=\emptyset$; if $a=T$ and $n\neq 12$ then Proposition~\ref{prop:eightemptyset} implies that $S^{v}(n)=\emptyset$; if $a=T$ and $n=12$ then it is routine to check that $S^{v'}(n)=\emptyset$. Thus we may assume that $c=F$. By Propositions~\ref{prop:FFFT} and~\ref{prop:TFFT} if $S^v(n)\cap S^{v'}(n)\neq \emptyset$ then $\Gamma_n(1,n/2-1)\in S^{v'}(n)$ where either $a=T$ and $n\equiv 0$~mod~$6$ or $a=F$ and $n\equiv 2,4$~mod~$6$, as required.
\end{proof}

Based on experiments in GAP that compare groups' abelianisations for $n\leq 500$, we conjecture that neither of the conclusions (a) nor (b) of Proposition~\ref{prop:nonintersection} are possible (which would then imply that $S^v(n)\cap S^{v'}(n)=\emptyset$ for $n>6, v\neq v'$):

\begin{conjecture}\label{conj:possibleisoms}
Let $n\geq 19$.
\begin{itemize}
  \item[(a)] If $n\equiv 0$~mod~$6$ then $\Gamma_n(1,n/2-1) \not \in S^{TFFF}(n)$;
  \item[(b)] if $n\equiv 2,4$~mod~$6$ then $\Gamma_n(1,n/2-1) \not \in S^{FFFF}(n)$.
\end{itemize}
\end{conjecture}

Question~5 of~\cite{BV} considers the groups~(\ref{eq:Gnmk}) and asks if it is possible to have $G_n(m,k)\cong G_{n'}(m',k')$ with $n\neq n'$. For the class of groups $\Gamma_n(k,l)$ the answer to the corresponding question is `yes' since (for example) $\Gamma_3(1,2)\cong\Gamma_6(1,2)\cong \Z*\Z$ and $\Gamma_3(1,2)\cong \Gamma_4(1,2)\cong \Z_3$. However, it is plausible that the answer is `no' if the groups involved are neither $\Z_3$ nor $\Z*\Z$. It is clear that for $n,n'\geq 3$ the group $\Z*\Z*\Z_{(2^{n/3}-(-1)^{n/3})/3} \cong \Z*\Z*\Z_{(2^{n'/3}-(-1)^{n'/3})/3}$ if and only if $n=n'$. Since $B((2^n-(-1)^n)/3,3,2^{2n/3},1)$ has order $2^n-(-1)^n$, we also have that $B((2^n-(-1)^n)/3,3,2^{2n/3},1)\cong B((2^{n'}-(-1)^{n'})/3,3,2^{2n'/3},1)$ if and only if $n=n'$. It follows from Corollary~\ref{cor:Gamman1n/2-1pairwisenonisom} that that if $\Gamma_n(1,n/2-1)\cong \Gamma_{n'}(1,n'/2-1)$ then $(n,6)=(n',6)=6$.

By Theorem~\ref{thm:ABCDsets} it remains to understand the sets $S^{FFFF}(n)$ and \linebreak $S^{TFFF}(n)$; we consider these in Sections~\ref{sec:FFFF} and~\ref{sec:210}.

\section{Elements of $S^{FFFF}(n)$}\label{sec:FFFF}

Proposition~\ref{prop:elementofFFFF} showed that  $f^{FFFF}(n)\geq 1$ for all $n\geq 10$. We now give improved lower bounds for $f^{FFFF}(n)$ when $n$ has certain factors. In order to distinguish certain groups in $S^{FFFF}(n)$ we first determine their abelianisations.

\begin{lemma}\label{lem:Gn(1,n/2)ab}
Let $n$ be even. Then $\Gamma_{n}(1,n/2)^{\mathrm{ab}}=\Z_{2^{n/2} -(-1)^{n/2}}$.
\end{lemma}

\begin{proof}
The $i$'th and $(i+n/2)$'th relators of $\Gamma_{n}(1,n/2)^{\mathrm{ab}}$ are $x_ix_{i+1}x_{i+n/2}$ and $x_{i+n/2}x_{i+n/2+1}x_{i}$, respectively. Together these imply that $x_{i+n/2+1}=$ $(x_ix_{i+n/2})^{-1}=x_{i+1}$ so $x_i=x_{i+n/2}$ for all $0\leq i< n$. Thus
\begin{alignat*}{1}
\Gamma_n(1,n/2)^\mathrm{ab}
&=\pres{x_0,\ldots ,x_{n-1}}{x_ix_{i+1}x_{i+n/2},\ x_i=x_{i+n/2}\ (0\leq i<n)}^\mathrm{ab}\\
&=\pres{x_0,\ldots ,x_{n/2-1}}{x_ix_{i+1}x_{i}\ (0\leq i<n/2)}^\mathrm{ab}\cong \Z_{2^{n/2} -(-1)^{n/2}}.
\end{alignat*}
\end{proof}

\begin{lemma}\label{lem:d(gamma_n(1,3))}
Let $8|n$. Then the minimum number of generators \linebreak $d(\Gamma_n(1,3))^\mathrm{ab})\geq 3$.
\end{lemma}

\begin{proof}
Here $n=8t$ for some $t\geq 1$ and so
\[\Gamma_n(1,3)=\pres{x_0,\ldots, x_{8t-1}}{x_ix_{i+1}x_{i+3}\ (0\leq i<8t)}\\
\]
which maps onto
\begin{alignat*}{1}
Q&=\pres{x_0,\ldots, x_{8t-1}}{x_ix_{i+1}x_{i+3},x_i=x_{i+8}\ (0\leq i<8t)}\\
&=\pres{x_0,\ldots, x_7}{x_ix_{i+1}x_{i+3}\ (0\leq i<8)}\\
&=\Gamma_8(1,3)
\end{alignat*}
and since $d(\Gamma_8(1,3)^\mathrm{ab})=3$ the result follows.
\end{proof}

\begin{corollary}\label{cor:SFFFFwith8|n}
Let $n\geq 16$ and $8|n$. Then $\Gamma_n(1,n/2)$, $\Gamma_n(1,3)$ are non-isomorphic elements of $S^{FFFF}(n)$, so $f^{FFFF}(n)\geq2$.
\end{corollary}

\begin{proof}
By Lemma~\ref{lem:Gn(1,n/2)ab} we have $d(\Gamma_n(1,n/2))=1$ and by Lemma~\ref{lem:d(gamma_n(1,3))} we have $d(\Gamma_n(1,3))^\mathrm{ab})\geq 3$ so $\Gamma_n(1,n/2)\not \cong\Gamma_n(1,3)$.
\end{proof}

Lengthy applications of the formula~(\ref{eq:Gaborder}), along the lines of the proof of~\cite[Theorem~4.4]{BW1}, give the following group orders (where $0$ indicates an infinite group):

\begin{lemma}\label{lem:abelianorders}
\begin{itemize}
  \item[(a)] If $4|n$ then $|\Gamma_n(1,n/4)^\mathrm{ab}|=|(2^{n/4}-(-1)^{n/4})(2^{n/4}+1-(-\sqrt{2})^{n/4}\cdot 2 \cos ((n-8)\pi/16)) |$;
  \item[(b)] if $6|n$ then $|\Gamma_n(1,n/6)^\mathrm{ab}|=|(2^{n/6}-(-1)^{n/6})(3^{n/6}+1-(-\sqrt{3})^{n/6}\cdot 2 \cos ((n-12)\pi/36))
  (2-(-1)^{n/6}\cdot 2 \cos (n-12)\pi/18   )            |$.
\end{itemize}
\end{lemma}
Together with Lemma~\ref{lem:Gn(1,n/2)ab} these allow us to obtain the following:

\begin{corollary}\label{cor:furthernonisoms}
\begin{itemize}
  \item[(a)] If $n\equiv 4,8,12$~mod~$16$ and $n\geq 12$ then $\Gamma_n(1,n/2)$, $\Gamma_n(1,n/4)$ are pairwise non-isomorphic elements of $S^{FFFF}$ and so $f^{FFFF}(n)\geq 2$;

    \item[(b)] if $n\equiv 0,24$~mod~$36$ then $\Gamma_n(1,n/2)$, $\Gamma_n(1,n/4)$, $\Gamma_n(1,n/6)$ are pairwise non-isomorphic elements of $S^{FFFF}$ and so $f^{FFFF}(n)\geq 3$.
\end{itemize}

\end{corollary}
Full details are available in the first named author's Ph.D. thesis~(\cite[Section~5.4]{EsamThesis}).

\section{The abelianisation as invariant}\label{sec:210}

In this section we assess how effective the abelianisation is as an invariant for proving non-isomorphism of non-isomorphic groups $\Gamma_n(k_1,l_1),\Gamma_n(k_2,l_2)$ where none of (B),(C),(D) hold for either presentation and where $n<210$. In this situation Theorem~\ref{thm:n<210} (below) identifies when non-isomorphic groups defined by pairs of parameters $(n,k_1,l_1),(n,k_2,l_2)$ cannot be distinguished by their abelianisations. For instance, it shows that if $\Gamma_{80}(k_1,l_1)^\mathrm{ab}\cong  \Gamma_{80}(k_2,l_2)^\mathrm{ab}$ and $\Gamma_{80}(k_1,l_1)\not \cong  \Gamma_{80}(k_2,l_2)$  then either
\begin{itemize}
\item[(i)] $\Gamma_{80}(k_i,l_i)\cong \Gamma_{80}(1,20)$ and $\Gamma_{80}(k_j,l_j)\cong \Gamma_{80}(1,40)$ ($\{i,j\}=\{1,2\}$); or
\item[(ii)] $\Gamma_{80}(k_i,l_i)\cong \Gamma_{80}(1,16)$ and $\Gamma_{80}(k_j,l_j)\cong \Gamma_{80}(1,17)$ ($\{i,j\}=\{1,2\}$).
\end{itemize}

The theorem was obtained and proved by the use of a computer program written in GAP. For each $n$ the program considers a set $S$ of pairs $(k,l)$ such that $(n,k,l)=1$, $1\leq k,l<n$, $k\neq l$, that represent all groups $\Gamma_n(k,l)$ for which none of (B),(C),(D) hold. If by using Lemma~\ref{lem:gammaiso} or Lemma~\ref{lem:gamman1liso} it is found that two pairs $(k_1,l_1),(k_2,l_2)\in S$ yield isomorphic groups $\Gamma_n(k_1,l_1),\Gamma_n(k_2,l_2)$ then one of the pairs is removed from $S$. This is repeated until $S$ is as small as possible. The abelianisation $\Gamma_n(k,l)^\mathrm{ab}$ is calculated for each remaining pair $(k,l)$ in $S$ and the pairs $(k_1,l_1)$, $(k_2,l_2)$ for which $\Gamma_n(k_1,l_1)^\mathrm{ab}\cong\Gamma_n(k_2,l_2)^\mathrm{ab}$ are recorded; these are the pairs presented in the statement of the theorem.

For efficiency reasons, results of this paper are used in defining the initial set $S$. The number 210 was chosen as it is the least integer with four prime factors so Corollary~\ref{cor:Gammap^aq^bLsets} could be invoked in many cases, and so reducing the scale of our computations.

\begin{theorem}\label{thm:n<210}
Let $3\leq n<210$, $1\leq k_i,l_i<n$, $k_i\neq l_i$, $B(n,k_i,l_i)=C(n,k_i,l_i)=D(n,k_i,l_i)=F$ ($\{i,j\}=\{1,2\}$). If $\Gamma_n(k_1,l_1)^\mathrm{ab}\cong\Gamma_n(k_2,l_2)^\mathrm{ab}$ and $\Gamma_n(k_1,l_1)\not \cong\Gamma_n(k_2,l_2)$ then $\Gamma_n(k_1,l_1)\cong \Gamma_n(k_1',l_1')$ and $\Gamma_n(k_2,l_2)\cong \Gamma_n(k_2',l_2')$, where $n$, $(k_1',l_1')$, $(k_2',l_2')$
satisfy one of the following:
\begin{itemize}
  \item[(0)] $(n,(k_1',l_1'),(k_2',l_2'))=(22,(1,4),(1,5)), (46,(1,7),(1,11))$;
  \item[(1)] $n\equiv 0$~mod~$16$, $\{(k_1',l_1'),(k_2',l_2')\}=\{(1,n/4),(1,n/2)\}$;
  \item[(2)] $n\equiv 0$~mod~$18$, $\{(k_1',l_1'),(k_2',l_2')\}=\{(1,n/3-2),(1,n/3+3)\}$;
  \item[(3)] $n\equiv 20,30$~mod~$50$, $\{(k_1',l_1'),(k_2',l_2')\}=\{(1,n/5),(1,n/5+1)\}$;
  \item[(4)] $n\equiv 0,40$~mod~$50$, $\{(k_1',l_1'),(k_2',l_2')\}=\{(1,n/5),(1,2n/5)\}$;
  \item[(5)] $n\equiv 10$~mod~$50$, $\{(k_1',l_1'),(k_2',l_2')\}=\{(1,n/5+1),(1,2n/5+1)\}$;
  \item[(6)] $n=2^{2\alpha+1}\cdot 3$ ($\alpha\geq 2$), $\{(k_1',l_1'),(k_2',l_2')\}=\{(1,n/12+1),(1,5n/12)\}$;
  \item[(7)] $n=2^{2\alpha}\cdot 3$ ($\alpha\geq 2$), $\{(k_1',l_1'),(k_2',l_2')\}=\{(1,n/12),(1,5n/12+1)\}$;
  \item[(8)] $n=2^\gamma\cdot 3^\beta$ ($\gamma\geq 4, \beta\geq 2$), $\{(k_1',l_1'),(k_2',l_2')\}=\{(1,n/12),(1,5n/12)\}$.
\end{itemize}
\end{theorem}

In some cases we have been able to prove non-isomorphism of the groups $\Gamma_n(k_1',l_1')$ and $\Gamma_n(k_2',l_2')$ arising in Theorem~\ref{thm:n<210}. For $n\leq 60$ the pairs of groups arising in Theorem~\ref{thm:n<210} are presented in Table~\ref{tab:nonisomgps60}, where we record when we have been able to prove non-isomorphism. (For $60<n<210$ we have not been able to prove non-isomorphism of the groups in any pair.) It was shown that $\Gamma_{18}(1,4)\not \cong \Gamma_{18}(1,9)$ by comparing the second derived quotients; it was shown that $\Gamma_{36}(1,11)\not \cong \Gamma_{36}(1,15)$ by showing that the former has an index 2 subgroup whereas the latter does not; in all other cases non-isomorphism was shown by comparing abelianisations of the groups' index~3 subgroups.

We further investigated pairs of groups $\Gamma_n(k_1',l_1'), \Gamma_n(k_2',l_2')$ corresponding to cases (1)--(8) for higher values of~$n$. For $n\leq 1000$ we have verified computationally that the groups in each pair have isomorphic abelianisations and that they cannot be proved isomorphic using Lemma~\ref{lem:gammaiso}. Therefore cases (1)--(8) provide a source of potentially non-isomorphic groups that cannot be distinguished by their abelianisations.

\begin{table}
\begin{center}
\begin{tabular}{cc}
  \begin{tabular}{|c|r@{\ }l|c|c|c|}\hline
  $n$ & $\{(k_1',l_1'),$ & $(k_2',l_2')\}$ &  Isom.\\\hline
  $10$ & $\{(1,3),$ & $(1,5)\}$ & No\\\hline
  $16$ & $\{(1,4),$ & $(1,8)\}$ & No\\\hline
  $18$ & $\{(1,4),$ & $(1,9)\}$ & No\\\hline
  $20$ & $\{(1,4),$ & $(1,5)\}$ & No\\\hline
  $22$ & $\{(1,4),$ & $(1,5)\}$ & No\\\hline
  $30$ & $\{(1,6),$ & $(1,7)\}$ & No\\\hline
  $32$ & $\{(1,8),$ & $(1,16)\}$ & ?\\\hline
  $36$ & $\{(1,11),$ & $(1,15)\}$ & No\\\hline
  \end{tabular}
  &
  \begin{tabular}{|c|r@{\ }l|c|c|c|}\hline
  $n$ & $\{(k_1',l_1'),$ & $(k_2',l_2')\}$ &  Isom.\\\hline
  $40$ & $\{(1,8),$ & $(1,16)\}$ & No\\\hline
  $46$ & $\{(1,7),$ & $(1,11)\}$ & No\\\hline
  $48$ & $\{(1,12),$ & $(1,24)\},$ & ?\\
  & $\{(1,4),$ & $(1,21)\}$ & ?\\\hline
  $50$ & $\{(1,10),$ & $(1,20)\}$ & No\\\hline
  $54$ & $\{(1,16),$ & $(1,21)\}$ & ?\\\hline
  $60$ & $\{(1,13),$ & $(1,25)\}$ & No\\
        &          &          &\\\hline
  \end{tabular}
  \end{tabular}
  \end{center}
\caption{Pairs of (possibly) non-isomorphic groups\label{tab:nonisomgps60}}
\end{table}

\section*{Acknowledgements}

We thank the referees of an earlier version of this paper for the helpful suggestions.

\bigskip

  \textsc{Department of Mathematics, Faculty of Education, Omar-Almukhtar University, Albyda, Libya.}\par\nopagebreak
  \textit{E-mail address}, \texttt{essamadeen@yahoo.com}

  \medskip

  \textsc{Department of Mathematical Sciences, University of Essex, Wivenhoe Park, Colchester, Essex CO4 3SQ, UK.}\par\nopagebreak
  \textit{E-mail address}, \texttt{Gerald.Williams@essex.ac.uk}

\end{document}